\documentclass[a4paper, oneside]{article}

\usepackage{dcolumn}
\usepackage{bookmark}
\usepackage{hyperref}

\newcolumntype{d}[1]{D{.}{.}{#1}}

 \usepackage[latin1]{inputenc}
 \usepackage[pdftex]{graphicx}
 \usepackage{amsfonts}
 \usepackage{amsthm}
 \usepackage{verbatim}
 \usepackage{amsmath}
 \usepackage{authblk}
 \usepackage[footnotesize]{caption}
 \usepackage{hyperref}
 \usepackage{natbib}
 \usepackage{enumitem}
 \usepackage{ragged2e}

\newtheorem{thm} {\textbf{Theorem}}

\newtheorem{defn} {\textbf{Definition}}
\newtheorem{ex} {\textbf{Example}}
\newtheorem{prop} {\textbf{Proposition}}
\newtheorem{prob} {\textbf{Problem}}

\newtheorem*{code} {\textbf{Code}}

\begin{document}

\title{A differential extension of Descartes' foundational approach: a new balance between symbolic and analog computation}

\author[1]{Pietro Milici}
\affil[1]{
Universit\'e de Bretagne Occidentale, Brest, France; \texttt{p.milici@gmail.com}}
\date{}
\maketitle

\begin{abstract}
	In \emph{La G\'eom\'etrie}, Descartes proposed a ``{balance}'' between geometric constructions and symbolic manipulation with the introduction of suitable ideal machines. In modern terms, that is a balance between analog and symbolic computation. 
	
	Descartes' geometric foundational approach (analysis without infinitary objects and synthesis with diagrammatic constructions) has been extended beyond the limits of algebraic polynomials in two different periods: by late 17th century \textit{tractional motion} and by early 20th century \textit{differential algebra}.	
	This paper proves that, adopting these extensions, it is possible to define a new convergence of machines (analog computation), algebra (symbolic manipulations) and a well determined class of mathematical objects that gives scope for a constructive foundation of (a part of) infinitesimal calculus without the conceptual need of infinity.
	To establish this balance, a clear definition of the constructive limits of tractional motion is provided by a \textit{differential universality theorem}.
\end{abstract}

\section{Introduction}\label{intro_ch}
In \emph{La G\'eom\'etrie}, Descartes proposed a ``{balance}'' between geometric constructions and symbolic manipulation with the introduction of suitable ideal machines. In particular, Cartesian tools were polynomial algebra (analysis) and a class of diagrammatic constructions (synthesis). This setting provided a classification of curves, according to which only the algebraic ones were considered ``purely geometrical.'' 
Thanks to this approach, geometrical intuition was no longer necessary in proving new properties, because the ``method'' (suitable algorithms in the analysis) permitted to lead the thought: in modern term, there was the seed of automated reasoning.
Descartes' limit was overcome with a general method by Newton and Leibniz introducing the infinity in the analytical part, whereas the synthetic perspective gradually lost importance with respect to the analytical one---geometry became a mean of visualization, no longer of construction.

Descartes' foundational approach (analysis without infinitary objects and synthesis with diagrammatic constructions) has, however, been extended beyond algebraic limits, albeit in two different periods. In the late 17th century, the synthetic aspect was extended by \textit{tractional motion} (construction of transcendental curves with idealized machines). In the first half of the 20th century, the analytical part was extended by \textit{differential algebra}, now a branch of computer algebra in which, informally speaking, the indeterminates are not numbers but continuous (and differentiable) functions.
This paper seeks to prove that it is possible to obtain a new balance between these synthetic and analytical extensions of Cartesian tools for a class of transcendental problems. 

A reason for a renewing of the Cartesian program concerns the historical evolution of mathematical objects. Mathematics can be considered as based on two cornerstones: arithmetics (symbolic manipulation of discrete elements) and geometry (constructions based on idealized continuous behaviors).
The synthetic components of such approaches are respectively digital and analog computation, and the mutual relationship between these kind of constructions provides a cognitive richness that leads the main steps of mathematical evolution. 
When arithmetic and geometric strengths are unbalanced, their mutual conversions can constitute a challenge to overreach their own limits, as evinced from a linguistic perspective in \cite[Ch. 1]{Kva2008} (even though without distinguishing between the constructive and mere visual power of geometry). Being today mainstream mathematics too much oriented toward arithmetics, our aim is to resume the commitment toward geometric constructions.
We are dealing with geometry instead of general analog computing because, according to Descartes' perspective, the primitive bases of our knowledge have to be intuitively clear, and the simple components of geometric ideal machines 
minimize the physical complexity and the cognitive requirements. 
From this perspective, the millennial endurance of Euclid's geometric paradigm can be justified by the wide application of its synthetic solutions, by the rigor of its analysis, but also by the concreteness of its constructive tools (segments and circles can be traced by ruler and compass).


Concerning the organization of this work, we show a new convergence of machines (analog computation, section \ref{tractional_section}), algebra (symbolic manipulations, section \ref{differential_algebra}), and geometry (constructed mathematical objects, section \ref{manifolds}) that, together with a problem solving method (section \ref{diff_problems}), gives scope for a foundation of (a part of) infinitesimal calculus without the conceptual need of infinity\footnote
{An objection to such avoidance of infinity could be that we cannot really avoid the infinite in the analytic part, because to define continuous functions at the basis of differential algebra we need limits or similar tools. With regard to this objection, we claim that, even if one considers continuity expressible only through infinitary tools, the allowed operations in differential algebra remain in the field of a finitist symbolic manipulation (in fact, differential algebra is nowadays considered a field of computer algebra). The constructive role of infinity in differential algebra is avoidable as it is in the analysis of polynomial algebra. In classical algebra, indeterminates assume values on the field of the real numbers, the definition of which requires infinity, but algebra remains finite because it does not deal with general real numbers, one simply makes manipulations and can control only a countable subset of real numbers. Similarly, differential algebra does not deal with the definition of continuous functions: the underlying requirement is to manipulate symbols that represent such functions.
}.
To establish this balance, a clear (historically missing) definition of the constructive limits of tractional motion is provided by a \textit{differential universality theorem}.

The peculiarity of this work lies in the attention to the constructive role of geometry as idealization of machines for foundational purposes.
This approach, after the \textit{de-geometrization} of mathematics, is far removed from the mainstream discussions of mathematics, especially regarding foundations.
However, though forgotten these days, the problem of defining appropriate canons of construction was very important in the early modern era, and heavily influenced the definition of mathematical objects and methods.
According to Bos' definition in \cite{Bos2001}, these are \textit{exactness problems} for geometry.

\section{Machines (analog computation)}\label{tractional_section}

\subsection{Brief history of tractional motion}\label{tractional_motion_section}
The problem of extending geometry beyond Cartesian limits was dominant between 1650 and 1750 \cite{Bos1988}, and in this section, we shortly deal with it.

If direct tangent problems are present since the classical period, it was only in the second half of the 17th century that the inverse ones 
appeared. The main difference between direct and inverse tangent problems is the role of the curve: in the direct case it is given \emph{a priori}, while in the second the curve is sought given some properties that its tangent has to satisfy.
Even though beyond Cartesian geometry, to legitimate solutions of inverse tangent problems there was the introduction of certain machines, intended as both theoretical and practical instruments, able to trace such curves.
The first documented curves constructed under tangent conditions were physically realized by the traction of a string tied to a load, which is why the study of these machines was named \textit{tractional motion}. 
Consider the following example:
\begin{quote}
	On a horizontal plane, a small heavy body (subjected to the friction on the plane) is tied with an ideally weightless non-elastic string, and imagine (slowly) pulling the other end of the string along a straight line drawn on the plane. 
\end{quote}	
Because of the friction on the plane, the body offers resistance to the pulling of the string: if the motion is slow enough to neglect inertia, the curve described by the body is called a \textit{tractrix}. 
Examining the left of Fig. \ref{tractrix}, we can see how the curve is traced thanks to the property that the string is constantly tangent to the curve.

\begin{figure}
	\center 
	\includegraphics[height=.2\textwidth] {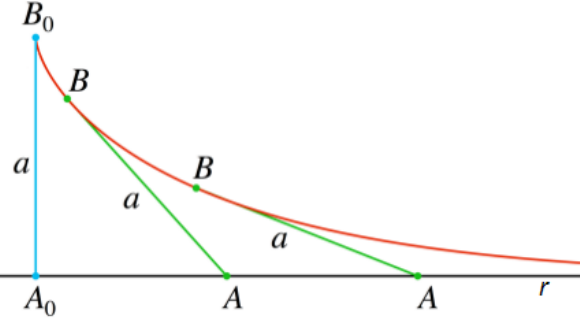}
	\hspace{1em}
	\includegraphics[height=.25\textwidth]{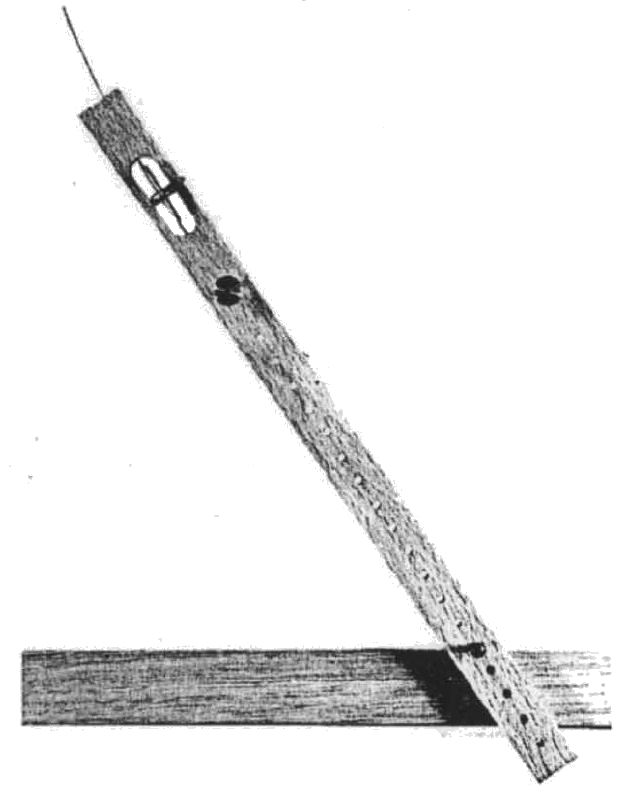}
	\caption[Tractrix]{
	[Left] The heavy body is $B$, with initial position $B_0$, the string is $a$, and the other end of the string is $A$, with initial position $A_0$. Moving $A$ along $r$, $B$ describes the tractrix (obviously, the movement is not reversible because of the non-rigidity of the string). Note how $a$ is tangent to the curve at every point.
	\newline
	[Right] Reconstruction of Perks' instrument for the tractrix \cite[p. 17]{Ped1963}, copyright license number 4518140353283. One can see the wheel taking the place of the load and a bar instead of the string.
	}\label{tractrix}
\end{figure}
During this period, mathematicians like Huygens began to consider instruments that, like the handlebars of a bike, could guide the tangent of a curve (in analytical mechanics terms, they introduced \textit{non-holonomic} constraints), in order to avoid inessential physical complications and so to consider tractional motion as ``pure geometric.'' Tractional motion suggested the possibility of constructing curves by imposing tangential conditions, generalizing (in a non-Cartesian way) the idea of geometrical objects, and constructing with new tools not only algebraic curves, but also some transcendental ones (seen as solutions of differential equations). During this period, the development of geometrical ideas often corresponded to the practical construction (or at least conception) of mechanical machines able to embody the theoretical properties, and thus able to trace the curves. Concerning practical machines, we recall those introduced in \cite{Per1706, Per1714} (for the machine for the tractrix see the right of Fig. \ref{tractrix}), which, for the first time, included a ``rolling wheel'' to guide the tangent. A more influential role for similar machines was played by the ones proposed in \cite[\textit{Ad Iacobum Hermannum Epistola}]{Pol1729}. 
An overview of such machines is visible in \cite{CM2019}.
As deepened below, the wheel is able to solve the inverse tangent problem because it avoids the lateral motion of its contact point. Furthermore, wheels imply less physical problems than dragging loads (e.g. inertia).

While questions about exactness in geometric constructions were so important in the early modern period, they disappeared in the 18th century because of the general affirmation of symbolic procedures, later considered autonomous from geometry. But, in contrast to what happened for algebraic curves, tractional motion did not reach a widely affirmed canon of constructions. Moreover, due to the change in paradigm, the geometric-mechanical ideas behind tractional machines remained forgotten for centuries, even for practical purposes, and were independently re-invented in the late 19th century, when they were used to build some grapho-mechanical instruments of integration (integraphs) to analogically compute symbolically non-solvable problems (for further reading, see \cite{Bla2017, Tou2009}).

\subsection{Components of tractional motion machines}
Leaving history behind, the goal of this part is to clearly define the components that can be used to obtain devices that implement certain tangent properties on a plane. Such clarification about the machines to be accepted in tractional constructions was historically missing: these ideal devices have only recently been defined with the so called \textit{tractional motion machines} (or TMMs) \cite{Mil2012, Mil2015}. 

We define the mechanical components that are allowed in a modern interpretation of tractional motion: we adopt these because they seem to give a good compromise between the simplicity of the components (two instruments and two constraints) and that of the assembled machines (even if the proposed components are not minimal \cite[Section 2]{Mil2012}). Machines obtained assembling these components (to be considered on a plane that can be infinitely extended) can be considered as an extension of Kempe's linkages:
\cite{Kem1876} stated the so-called Kempe's \textit{universality theorem}, that every bounded portion of a planar algebraic curve can be traced by linkages made of jointed finite-length rods (the proof was flawed, a corrected proof is in \cite{KM2002}).
In section \ref{constructible_functions} we provide a generalization of Kempe's result for TMMs with a \textit{differential universality theorem}.
Now, let us define the components of \textit{Tractional Motion Machines}.

\begin{itemize}
	\item 	We adopt \textbf{rods}, and assume these have perfect straightness and negligible width. 
	They can be finite or infinitely extensible: in both cases they are different from the Euclidean segments and straight lines, because they are not statically traced objects but planar rigid bodies (mechanical entities with three degrees of freedom, two characterizing the position of a specific point and the third identifying the slope with respect to a fixed line). 
	
	\item	 It is possible to put some \textbf{carts} on a rod, each one using the rod as a rail: a cart has one degree of freedom once placed on a rod (the cart can only move up and down the rod).
	
	\item 	The \textbf{joint} is a constraint between fixed points of two (or more) different objects (here, ``object'' refers to the plane, a rod, or a cart). Once the joint has been applied, jointed objects can only rotate around their common point (note that, in general, the junction point does not have to be fixed on the plane).
	
	\item 	Finally, we have the non-holonomic constraint, the \textbf{wheel}: once a rod $r$ and a point $S$ on $r$ have been selected, we can set a wheel at $S$ that prevents $S$ itself moving perpendicularly to $r$ (considering the motion of $S$ with respect to the plane). Technically, this is as if we put a fixed caster (oriented like $r$) at $S$, with its wheel rotating without slipping on the plane. As evinced since the construction of the tractrix, the avoidance of lateral motion in the rod at a point is strongly related to the tangent. If we consider the caster wheel as a disk rolling perpendicularly to the base plane, the projection of the disk surface is always tangent to the curve described by the disk contact point (see Fig. \ref{rolling_wheel}). Thus, the rod is tangent to the curve traced by the wheeled point, having the same direction as the caster wheel.
	
\end{itemize}

\begin{figure}
	\center \includegraphics[scale=.5] {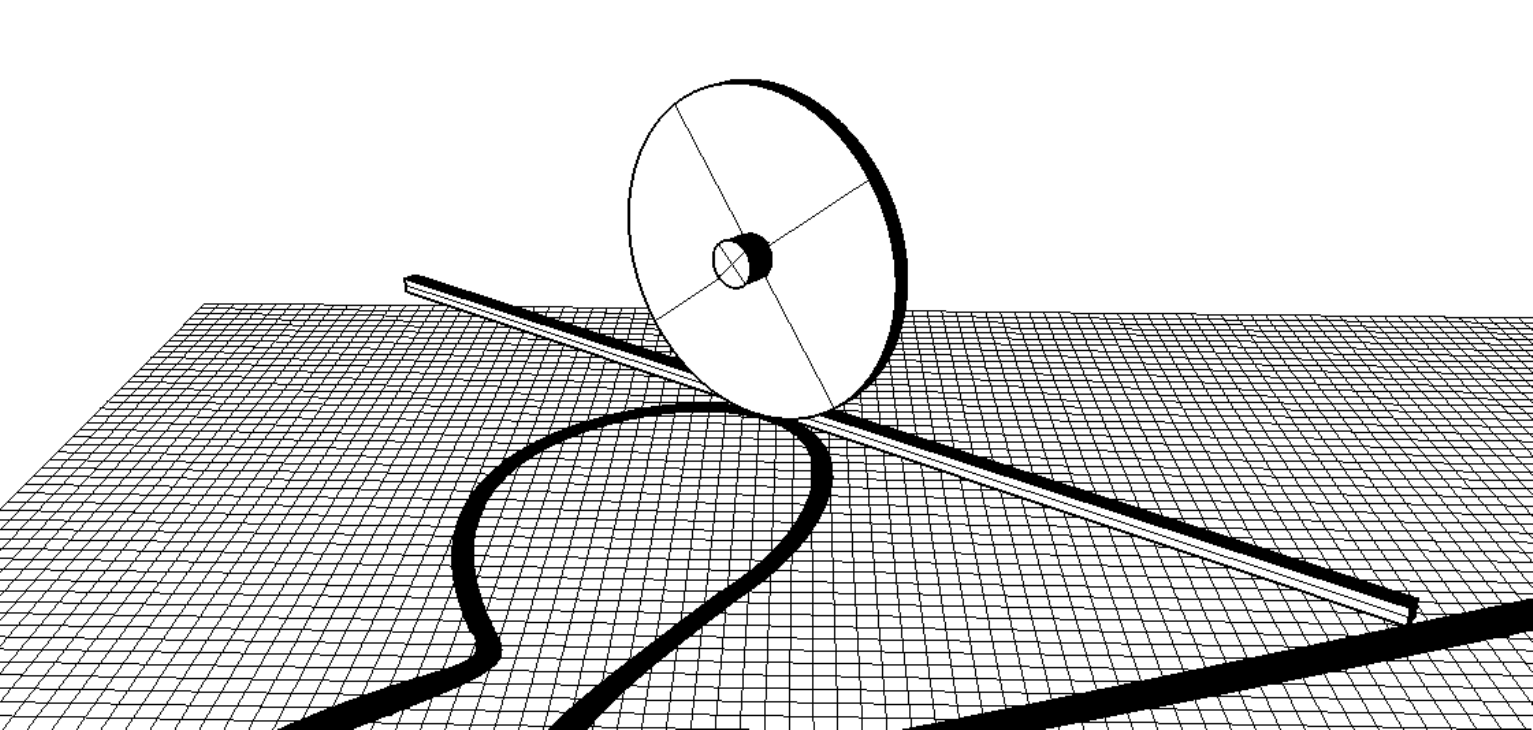}
	\caption{A wheel rolling while following any regular curve has the property that its own \textit{direction} (represented in the picture by a bar) is always tangent to the curve.}\label{rolling_wheel}
\end{figure}

Like Kempe's linkages, even TMM tools are assumed to be ideal (we do not care about physical inaccuracies), and we do not consider problems related to collisions of rods or of different carts on the same rod. Once specified such details, these components can be used to assemble machines whose motion on a plane is purely kinematic (just kinematic constraints, with no attention to other physical interrelations). For a diagrammatic representation of assembled components, see Fig. \ref{components}. 

As an important remark note that, differently from the general setting of linkages, we are not only looking for a mechanical method to define geometrical objects, but for a computational model not involving infinity from a foundational perspective. That means that we cannot accept any general distance between points fixed on a rod, because that would imply the introduction of real numbers and so of uncountable sets. As deepened in the following section, a simple solution is to introduce an arbitrary unit length and, for any point $P$ fixed on a rod $r$, to admit the constructability of the points on $r$ distant one unit from $P$. This unit-distance primitive operation is necessary because in our model there is no compass available to transfer lengths.

\begin{figure}
	\center \includegraphics[scale=0.5] {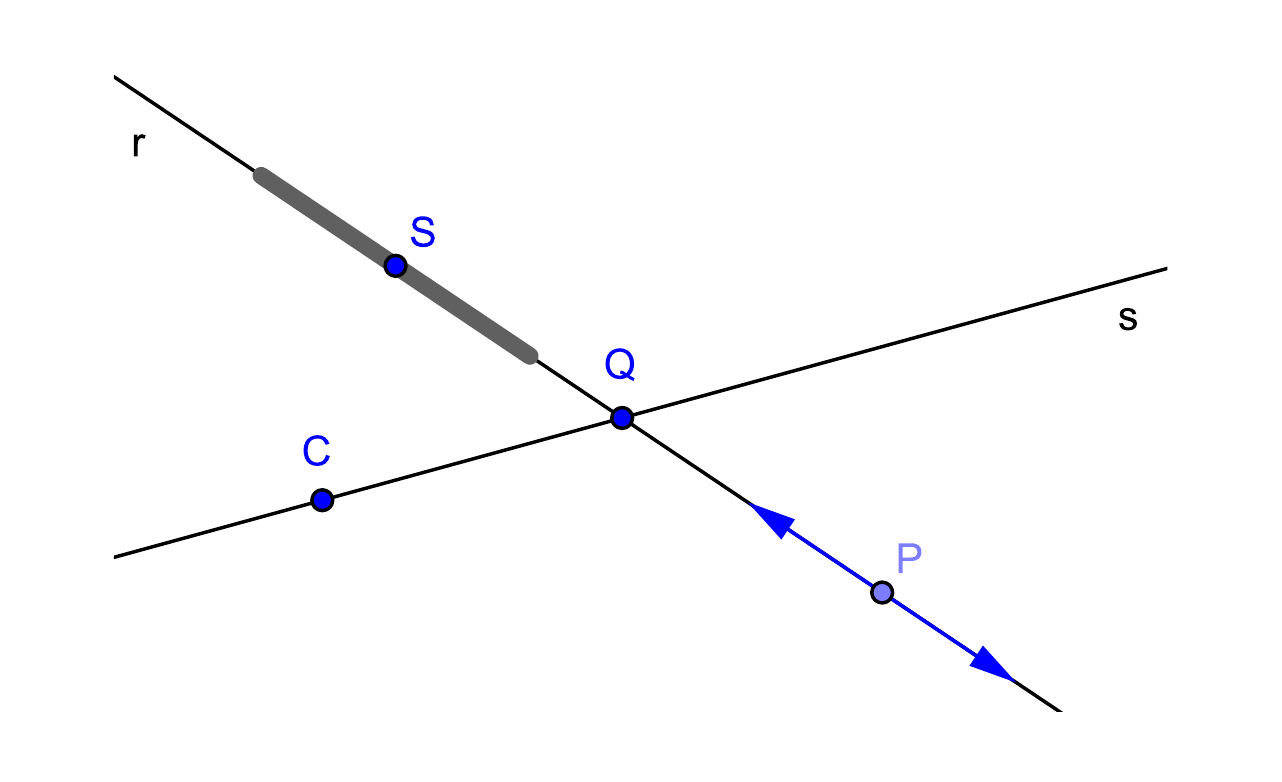}
	\caption{Schematic representation of the components: there are two rods ($r$ and $s$) joined at $Q$. On $r$, there is also a cart $P$ (the arrows stand for the possible motions the cart can have) and a wheel $S$ (the gray thick line ideally represents the projection of a wheel).}\label{components}
\end{figure}

\subsection{A code for constructions}\label{language}
Once introduced the components, we have to define how to properly assemble them with an adequate language (as an exemplary formalization of geometric constructions in a computational language see \cite{Huc1989}). 
In our language, points are enumerated by natural numbers ($P_0, P_1, \ldots$), and the \textit{index} of the point $P_n$ is $n$. These points, if not differently imposed by other constraints, can freely move on the plane. However, to construct a TMM, we have to start from a certain number of given points that are fixed on the plane. For a minimal definition, we consider as given two distinct fixed points $P_0$ and $P_1$: we introduce the unit as their distance. 

Rods are represented by couples of natural numbers: by $r(i,j)$ we mean that a rod is introduced in the point $P_i$. However, we can consider many rods jointed in the same point, hence the number $j$ distinguishes between them. E.g. let $r$ be the first rod in our construction to be jointed in $P_3$, $s$ the second rod jointed in $P_3$: in our language $r$ is represented by $r(3,0)$ and $s$ by $r(3,1)$. In $r(i,j)$ the ordered values ($i,j$) are named the \textit{indices} of the rod. 
\\

There are three admissible instructions:
\begin{itemize}
    \item \texttt{onRod(k,i,j)} with $k,i,j\in \mathbb N$. It imposes that the point $P_j$ lies on the rod $r(k,i)$. Recalling the components, it implies to put a cart on a rod ($P_j$ can move along $r(k,i)$). 
    \item \texttt{dist(k,i,n,j)} with $k,i,j\in \mathbb N, n\in\mathbb Z$. It imposes that the point $P_j$ lies on the rod $r(k,i)$ at a distance $n$ from $P_k$. We have to note that $n$ can also be a negative integer, thus we need to define an orientation of the rod, as deepened below. This instruction implies the introduction of a joint.
    \item \texttt{wheel(k,i)} with $k,i\in \mathbb N$. It imposes a wheel on the point $P_k$ oriented as the rod $r(k,i)$.
\end{itemize}
We have to spend few more words about the introduction of integer distances on rods by \texttt{dist}. 
The starting point is that, once introduced the unit, we can constrain any two points on a rod to be at the distance of a unit.
Hence, we could consider the instruction \texttt{unitDist(k,i,j)} with $k,i,j\in \mathbb N$ that imposes $P_j$ to stay on $r(k,i)$ at a distance of one unit from $P_k$. Without the more general \texttt{dist}, we could obtain the point $P_l$ on $r(k,i)$ at a distance of 2 units from $P_k$ by introducing a new rod $r(j,0)$ with $P_k$ lying on it, and then imposing \texttt{unitDist(j,0,l)}. In this way, the point $P_l$ not coincident with $P_k$ is the sought point. 
Iterating this construction, we can consider any integer distance on a rod; however, that would require to 
impose conditions on 
coincidence of different points
(e.g. $P_l \not\equiv P_k$),
thus we prefer to adopt \texttt{dist} with any integer distance and an orientation.

Furthermore, about the possibility of introducing an orientation, we have to precise that the specific orientation of a rod is not important, it is important that the orientation of the rod is coherent with all its points: e.g. given the instructions \texttt{dist(0,0,2,2)} and \texttt{dist(0,0,-1,3)}, the distance between $P_2$ and $P_3$ has to be $3$, not $1$. As an example to practice with \texttt{dist}, consider in Fig. \ref{Watt} the visual representation of the following code (commented on the right). Note that, using only \texttt{dist}, we get exactly Kempe's linkages with integer-length rods.
\begin{ex}\label{linkages}
\begin{code}
    Considering the two given fixed points $P_0, P_1$, let us introduce the points $P_2, P_3$ at unary distance to, respectively, $P_0$ and $P_1$. Set the distance between $P_2$ and $P_3$ equal to 2 units, and consider their middle point $P_4$.
    
    	\begin{tabular} {p{0.5cm} p{3.5cm} p{9.5cm}} 
		& \texttt{dist(0,0,1,2)} & consider $P_2$ on the rod $r(0,0)$ at unary distance from $P_0$\\
		& \texttt{dist(1,0,1,3)} & consider $P_3$ on the rod $r(1,0)$ at unary distance from $P_1$\\
		& \texttt{dist(2,0,2,3)} & set the distance between $P_2$ and $P_3$ equal to 2 units along $r(2,0)$\\
		& \texttt{dist(2,0,1,4)} & $P_4$ is the point on $r(2,0)$ distant 1 unit from $P_2$\\
		\noalign{\vskip 1mm}
	\end{tabular}
\end{code}
\end{ex}
Note that, even though there are two point on $r(2,0)$ distant 1 unit from $P_2$, $P_4$ is uniquely defined because it has to be in the ray  $\overrightarrow{P_2P_3}$ ($P_3$ has distance $+2$ from $P_2$, thus with the same sign of the distance of $P_4$). The other point one unit away from $P_2$ is $P_5$ defined by \texttt{dist(2,0,-1,5)}.
About indices, note that two points with different indices can coincide: adding \texttt{dist(3,0,2,2)} and \texttt{dist(3,0,1,6)} would imply that the point $P_6$  concides with $P_4$ in every configuration.
\\
\begin{figure}
    \centering
    \includegraphics[width=.4\textwidth]{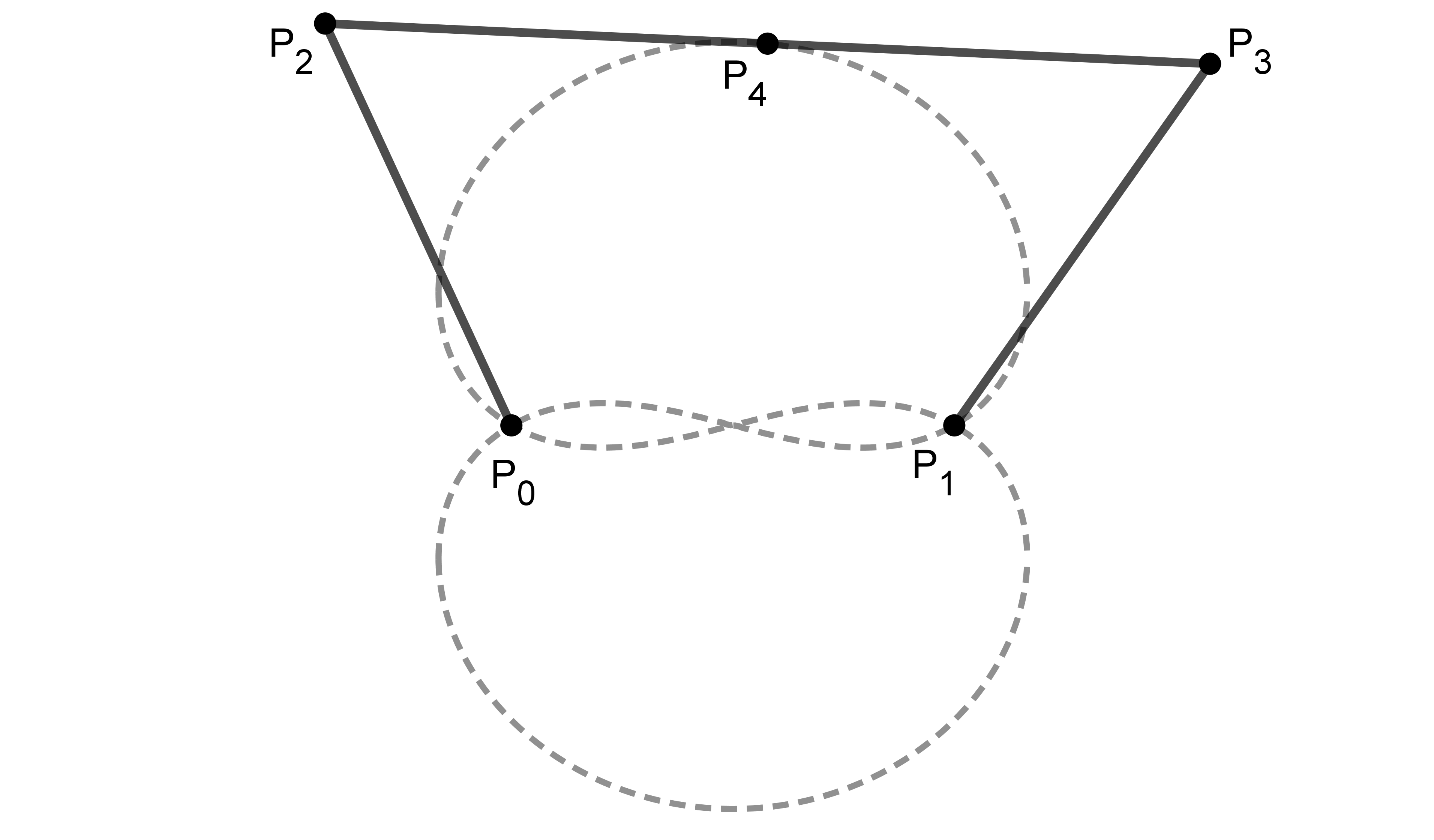}
    \caption{A simple machine without carts and wheels. The dotted line represents the locus defined by the point $P_4$. See footnote \ref{note:realAlg} at page \pageref{note:realAlg} for an analytical study.}
    \label{Watt}
\end{figure}

Once defined the instructions, we can consider their analytical conversion. That permits us to use computer algebra as a tool of automated reasoning on the behavior of our machines.
First of all, when introducing a rod $r(k,i)$, besides $P_k$ we have to consider an auxiliary point $Q_{k,i}$ to determine the orientation of the rod. This point has to satisfy only the property of being one unit away from $P_k$ (informally, $Q_{k,i}$ is the point $P_j$ s.t. \texttt{dist(k,i,1,j)}). Analytically, considering a system of Cartesian coordinates s.t. $P_0=(0,0)$ and $P_1=(1,0)$, and adopting the notation $x(P), y(P)$  respectively for the abscissa and the ordinate of the point $P$,  we can introduce $Q_{k,i}$ by the equation:
\begin{equation*}\label{eqQ}
    (x(P_k)-x(Q_{k,i}))^2+(y(P_k)-x(Q_{k,i}))^2=1.
\end{equation*}
With $Q_{k,i}$, it is easy to convert instructions in equations:
\texttt{onRod(k,i,j)} becomes 
\begin{equation*}\label{eqRod}
(x(P_k)-x(Q_{k,i}))(y(P_j)-y(Q_{k,i}))=(x(P_j)-x(Q_{k,i}))(y(P_k)-y(Q_{k,i}))
\end{equation*}
and \texttt{dist(k,i,n,j)} is expressed by
$$\begin{cases}\label{eqDist}
x(P_j)=x(P_k) + n (x(Q_{k,i})-x(P_k))\\
y(P_j)=y(P_k) + n (y(Q_{k,i})-y(P_k))
\end{cases}.$$

For the wheel constraint, we have to consider a point not only as a couple of coordinates, but as a couple of functions.
As typical in physics, consider $P_k=(x_k(t),y_k(t))$, i.e. consider the Cartesian coordinates of the point in function of the time. The instruction \texttt{wheel(k,i)}  poses the condition that $P_k$ cannot move perpendicularly to $r(k,i)$: so, considering $P_k'=\left(\frac{d}{dt}x_k,\frac{d}{dt}y_k\right)$, $P_k'$ has to be parallel to $Q_{k,i}-P_k$.
Thus, omitting the dependence on $t$ and considering $P_k'=(x'(P_k),y'(P_k))$, the wheel constrain becomes
\begin{equation*}\label{eqWheel}
y'(P_k) (x(Q_{k,i}) - x(P_k))=x'(P_k) (y(Q_{k,i})-y(P_k)).
\end{equation*}

Thus, both wheel constraints and the other instructions are translatable in polynomials in the variables and their derivatives: as deepened in the algebraic part (section \ref{differential_algebra}), such polynomials are named \textit{differential polynomials} and constitute the basis for differential algebra.
\\

In the following section we introduce subroutines to solve some problems about constructions. When using subroutines, we also need to define the instructions \texttt{return(i)} and \texttt{newPoint(i)}  (with $i \in \mathbb N$). Note that, once fixed the input, each implementation can be obtained without subroutines and the last two instructions (they are not adding any primitive to the model). However, subroutines are useful to give general constructions in function of inputs.

About \texttt{return(a)}, it returns the natural number $a$ to the program calling the subroutine.

The other instruction, \texttt{newPoint(i)}, returns a natural number not previously used to enumerate an already introduced point
(i.e. first and third arguments of \texttt{onRod}, first and forth arguments of \texttt{dist}, first argument of \texttt{wheel}). 
To stay compact, the shorter notation \texttt{[i]} stands for \texttt{newPoint(i)}.
This instruction comes in useful because, in a subroutine, we don't know which natural numbers are yet free to introduce new points.
Let $n$ be the highest value used as index of a point before the execution of the instruction:  \texttt{newPoint} returns the integer $n+1$. But later we may need to recall the newly introduced index, and we can use multiple \texttt{newPoint}: that's why it was necessary to add the index $i$ (e.g. \texttt{onRod([0],0,[1])} constrains the point $P_{[1]}$ to be on the rod $r([0],0)$).
Note that, to introduce a new rod in a given point $P_i$ without knowing how many rods have already been joined in it, we can consider a new point $P_{[k]}$ (assuming that previously we have introduced $P_{[0]}, P_{[1]}, \ldots, P_{[k-1]}$) coincident with $P_i$. The new rod $r([k],0)$ can be defined by the instruction \texttt{dist([k],0,0,i)}. 


To conclude, we have to remark that the use of \texttt{newPoint} 
can create some unwanted situations. For example, in a subroutine, it can happens that \texttt{newPoint} creates an index that is used in the following instructions as a fixed index of a point. To avoid that, it might be useful to have a set of indices reserved for these local construction points, 
but we leave such possible optimization (as a real software implementation) to future works.

\subsection{Algebraic machines}\label{Algebraic_machines}
It is also interesting to consider TMMs without wheels. In this case every constraint can be translated in algebraic polynomial equations, so we call the obtained machines ``algebraic.'' There are algebraic machines that cannot be considered as Kempe's linkages, because rods are not only used to constrain a fixed length between two junction points, but also to allow a point to move along a straight line (thanks to carts). 
	{With only Kempe's linkages to trace a straight line is not trivial at all, the problem was solved in an approximated way by Watt, and later exactly solved by Peaucellier (see, for example, \cite[pp. 29-30]{DOR2007}).
	Algebraic machines can be somehow considered as more adherent to Descartes' machines for geometry than to Kempe's linkages (Descartes used machines in which straight components were allowed to slide along other straight components).
	Furthermore, the introduction of ``extensible'' rods, allows to trace not only finite part of algebraic curves, but whole continuous branches of algebraic curves (as already observed, Kempe's linkages can construct only bounded portions of algebraic curves).
}

Before solving some problems to perform sum and product with algebraic machines, we need to remark that the components of algebraic machines are not part of Euclid's geometry: they are movable mechanical parts, not static traces on a plane, so constructions have to work dynamically. Thus, even though problems like the following ones can be easily solved with ruler and compass if we substitute ``rod'' with ``segment,'' we need new constructions for our mechanical setting.
All the constructions dynamically work in every configuration of the inputs, and in general input objects are points and rods that can move.

\begin{prob}[perpendicular]\label{perp}
	Given a rod $r$ and a point $P$, construct a rod $s$ perpendicular to $r$ passing through $P$.
\end{prob}
\begin{code}
    Let the rod $r$ and $P$ be in our language respectively $r(k,i)$ and $P_j$. The subroutine returns the number $a$ such that $r(a,0)$ is the sought rod. Note that we can guarantee that in the rod $r(a,0)$ the second index is 0 because the point $a$ is newly introduced by the subroutine.
    
    	\begin{tabular} {p{0.1cm} p{7.5cm} p{6.5cm}} 
		\textbf{\texttt{perp(k,i,j)}} &	& signature of the subroutine\\
		\noalign{\vskip 1mm}
		& \texttt{onRod(k,i,[0])}	& consider a new point $P_{[0]}$ on $r(k,i)$\\
		\noalign{\vskip 1mm}
		& \texttt{dist([0],0,4,[1])} & consider a new point $P_{[1]}$ on $r([0],0)$ distant 4 units from $P_{[0]}$ \\
		\noalign{\vskip 1mm}
		& \texttt{dist([0],1,3,[2])} & consider a new point $P_{[2]}$ distant 3 units from $P_{[0]}$\\
		\noalign{\vskip 1mm}
		& \texttt{dist([2],0,5,[1])} & constrain $P_{[1]}$ to stay 5 units away from $P_{[2]}$\\
		\noalign{\vskip 1mm}
		& \texttt{onRod([0],0,j)} & constrain $P_j$ to stay on $r([0],0)$\\
		\noalign{\vskip 1mm}
		& \texttt{onRod(k,i,[2])} & constrain $P_{[2]}$ to stay on $r(k,i)$\\
		\noalign{\vskip 1mm}
		& \texttt{return([0])}	& the sought rod is $r([0],0)$, return $[0]$\\
		\noalign{\vskip 1mm}
	\end{tabular}
\end{code}
\begin{proof}
	Construct a right triangle by the junction of a Pythagorean triple as rod lengths (e.g. connect three rods respectively of length 3-4-5). Consider an infinitely extensible rod $s$ for one of the catheti, and make the other cathetus slide on $r$ with two carts on the vertices. Pose another cart on $s$ in correspondence of $P$. As visible in Fig. \ref{perp_rod+}, that solves the problem.
	\begin{figure}
		\center \includegraphics[scale=.5] {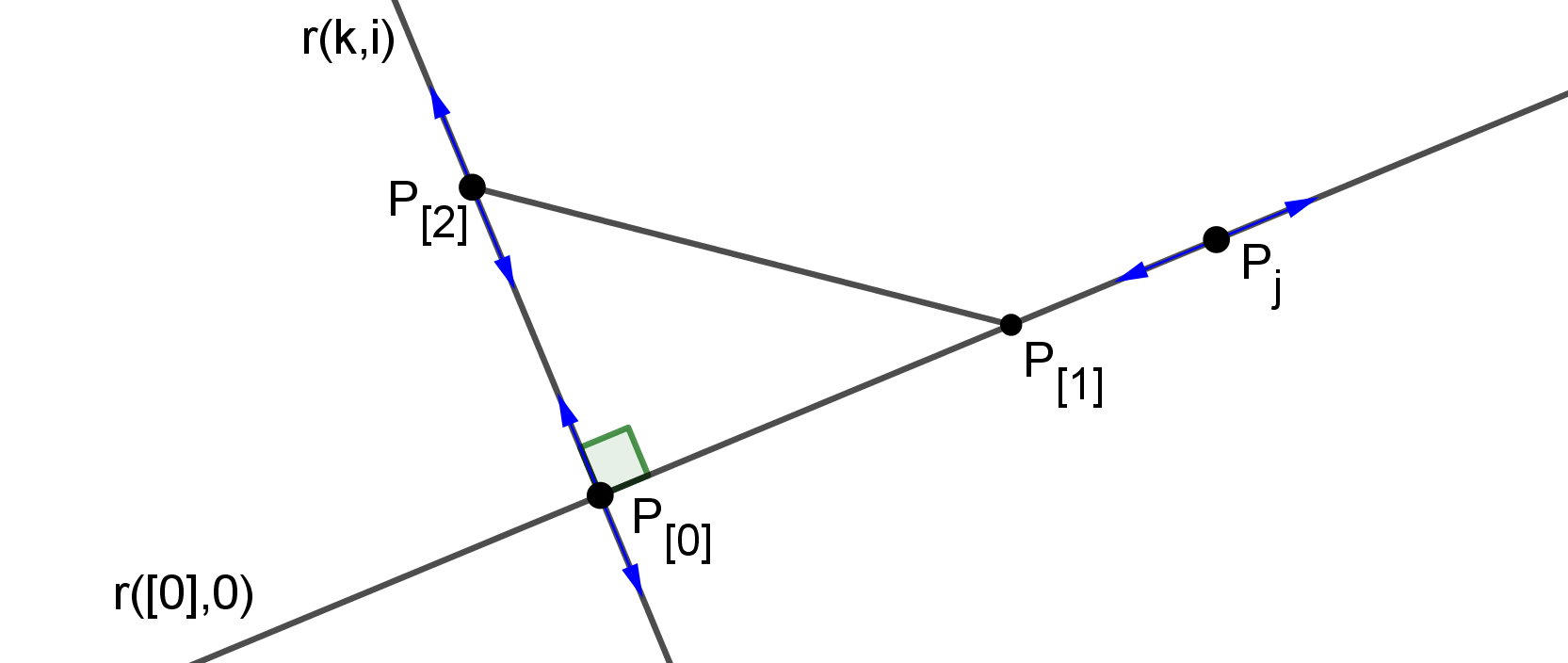}
		\caption{Construction of the rod $r([0],0)$ perpendicular to $r(k,i)$ passing through $P_j$.}\label{perp_rod+}
	\end{figure}
\end{proof}

\begin{prob}[parallel]\label{parallel}
	Given a rod $r$ and a point $P$, construct a rod parallel to $r$ passing through $P$.
\end{prob}
\begin{code}
    Let the rod $r$ and $P$ be respectively $r(k,i)$ and $P_j$. The subroutine returns the number $b$ such that $r(b,0)$ is the sought rod.
    
    	\begin{tabular} {p{0.1cm} p{6.5cm} p{7.5cm}} 
		\textbf{\texttt{parall(k,i,j)}} &	& signature of the subroutine\\
		\noalign{\vskip 1mm}
		& \texttt{return(perp(perp(k,i,j),0,j))} & let's start from the inner subroutines: $perp(k,i,j)$ returns the index $a$ s.t. $r(a,0)$ is a rod perpendicular to $r(k,i)$ passing through $P_j$. Then, $perp(a,0,j)$ returns the index $b$ s.t. $r(b,0)$ is a rod perpendicular to $r(a,0)$ passing through $P_j$. Such $b$ is returned. \\
		\noalign{\vskip 1mm}
	\end{tabular}
\end{code}
\begin{proof}
	According to Problem \ref{perp} we can construct a rod $s$ perpendicular to $r$ passing through $P$, and similarly we can construct a rod $t$ perpendicular to $s$ passing through $P$: $t$ solves the problem.
\end{proof}

In the following problems we adopt Cartesian coordinates to simplify the notation. 
We have already introduced the points $P_0$ and $P_1$ respectively of coordinates $(0,0)$ and  $(1,0)$. Let's consider the rod $r(0,0)$ as the usually oriented x-axis (by \texttt{dist(0,0,1,1)}), and the point $P_2$ at unary distance from $P_0$ on the rod perpendicular to $r(0,0)$ passing through $P_0$ (by \texttt{dist(perp(0,0,0),0,1,2)}). 
Considering the x-axis oriented horizontally to the right, the possible positions of $P_2$ can be above or below the x-axis. However, our constructions works for both the possibilities,  it only changes the orientation of the coordinate system. In the following figures we always consider the standard positive orientation (while the positive x-axis points right, the positive y-axis points up), but all the results remain with the negative orientation.

\begin{prob}[inverse]\label{inverse}
	Given a point of Cartesian coordinates $(x,y)$, construct a point of coordinates $(y,x)$.
\end{prob}
\begin{code}
    Assuming that $P_0,P_1,P_2$ are respectively $(0,0), (1,0), (0,1)$, let $P_i=(x,y)$ be the input point. The subroutine returns the index of the point $(y,x)$.
    
    	\begin{tabular} {p{0.1cm} p{7cm} p{7cm}} 
		\textbf{\texttt{inv(i)}} &	& signature of the subroutine\\
		\noalign{\vskip 1mm}
		& \texttt{onRod(2,0,1)} & the rod $r(2,0)$ passes through $P_1$\\
		\noalign{\vskip 1mm}
		& \texttt{onRod(perp(0,0,i),0,[0])} & $P_{[0]}$ is on the rod through $P_i$ parallel to the y-axis\\
		\noalign{\vskip 1mm}
		& \texttt{onRod(perp(2,0,0),0,[0])} & $P_{[0]}$ also is constrained to lie on the bisector of the first and third quadrant\\
		\noalign{\vskip 1mm}
		& \texttt{onRod(parall(0,0,[0]),0,[1])} & $P_{[1]}$ is on the rod through $P_{[0]}$ parallel to the x-axis\\
		\noalign{\vskip 1mm}
		& \texttt{onRod(parall(2,0,i),0,[1])} & $P_{[1]}$ is also constrained to lie on the rod parallel to $r(2,0)$ through $P_i$\\
		\noalign{\vskip 1mm}
		& \texttt{return([1])} & the sought point is $P_{[1]}$\\
		\noalign{\vskip 1mm}
	\end{tabular}
\end{code}
\begin{proof}
    As visible in Fig. \ref{fig:symm}, starting from the point $(x,y)$ and calling $r(2,0)$ the rod passing through $(1,0)$ and $(0,1)$, we can consider the points $(x,x)$ and finally $(y,x)$ by intersecting rods parallel and perpendicular to the x-axis and to $r(2,0)$.
\end{proof}
\begin{figure}
    \centering
    \includegraphics[width=.35\textwidth]{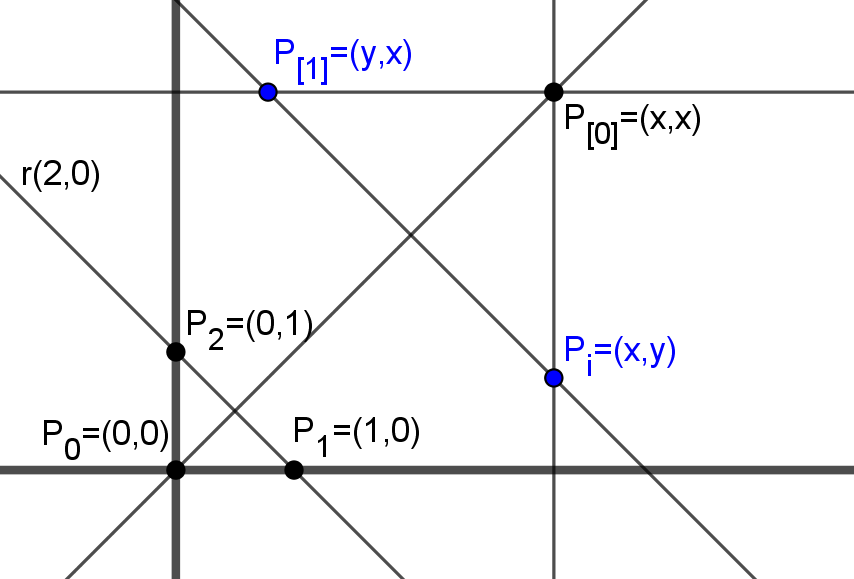}
    \caption{Schema of the construction of the point $(y,x)$ given the point $(x,y)$ using rods parallel and perpendicular to the x-axis and to $r(2,0)$ (rod through $(1,0)$ and $(0,1)$).}
    \label{fig:symm}
\end{figure}

According to Problem \ref{perp} we can project a point on x- and y- axes and, according to the last problem, we can transpose a length from the ordinate to the abscissa. 
Indeed, given a point of coordinates $(x_0,y_0)$, we can construct the points of coordinates $(x_0,0)$ and $(y_0,0)$. Conversely, given the points of coordinates $(x_0,0)$ and $(y_0,0)$, we can construct the point $(x_0,y_0)$. Thus, to represent variables in algebraic machines, we can interpret them simply as points moving on the abscissa. Algebraically, such variables are real values.
It's time to show how to perform the internal binary operations of sum, difference and multiplication for abscissas of points.

\begin{prob}[sum]\label{sum}
	Given two points of Cartesian coordinates $(x_1,y_1)$ and $(x_2,y_2)$, construct a point of coordinates $(x_1+x_2,0)$.
\end{prob}
\begin{code}
    Consider the x-axis $r(0,0)$, and let $(x_1,y_1), (x_2,y_1)$ respectively be $P_i, P_j$. The subroutine returns the index of the point $(x_1+x_2,0)$.
    
    	\begin{tabular} {p{0.1cm} p{7.5cm} p{6.5cm}} 
		\textbf{\texttt{sum(i,j)}} &	& signature of the subroutine\\
		\noalign{\vskip 1mm}
		& \texttt{onRod(parall(0,0,inv(i)),0,[0])} & a new point $P_{[0]}$ lies on $y=x_1$\\
		\noalign{\vskip 1mm}
		& \texttt{onRod(perp(0,0,j),0,[0])} &  $P_{[0]}$ has coordinates $(x_2,x_1)$\\
		\noalign{\vskip 1mm}
		& \texttt{onRod(inv(i),0,i)}
		& a new rod 
		$r(inv(i),0)$ passes through $(x_1,y_1)$ and $(y_1,x_1)$\\
		\noalign{\vskip 1mm}
		& \texttt{onRod(parall(inv(i),0,[0]),0,[1])} 
		& the new point $P_{[1]}$ lies on the rod parallel to 
		$r(inv(i),0)$
		passing through $P_{[0]}$\\
		\noalign{\vskip 1mm}
		& \texttt{onRod(0,0,[1])} & the point $P_{[1]}$ is constrained to lie on the x-axis\\
		\noalign{\vskip 1mm}
		& \texttt{return([1])} & the sought point $P_{[1]}$ has coordinates $(x_1+x_2,0)$\\
		\noalign{\vskip 1mm}
	\end{tabular}
\end{code}
\begin{proof}
	According to Problem \ref{inverse}, consider the point of coordinates $(y_1,x_1)$. Thanks to Problem \ref{parallel} we can consider the rod $r$ parallel to the x-axis passing through $(y_1,x_1)$ and the rod $s$ parallel to the y-axis passing through $(x_2,y_2)$. By carts we can identify the point $(x_2,x_1)$ (the one lying on both $r$ and $s$). Finally we can consider the rod $t$ joined in $(x_2,x_1)$ parallel to the rod passing through $(x_1,y_1)$ and $(y_1,x_1)$: the point in the intersection of $t$ and the x-axis has coordinates $(x_1+x_2,0)$ (cf. Fig. \ref{sum_rod}).
	\begin{figure}
		\center \includegraphics[width=.4\textwidth] {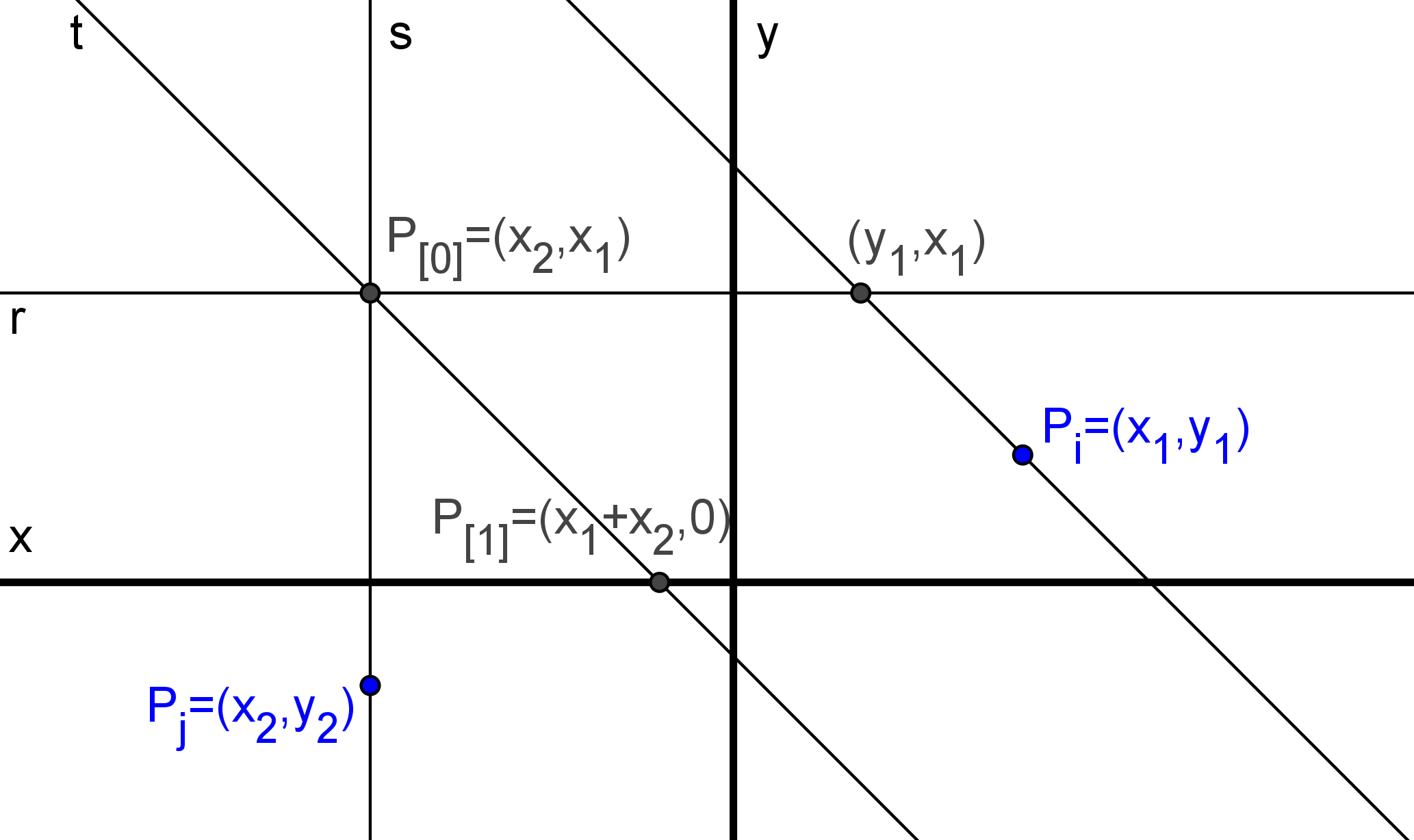}
		\caption[Construction for the sum]{Construction of the point $(x_1+x_2,0)$ given the points $(x_1,y_1)$ and $(x_2,y_2)$.}\label{sum_rod}
	\end{figure}
\end{proof}

\begin{prob}[difference]\label{diff}
	Given two points of coordinates $(x_1,y_1)$ and $(x_2,y_2)$, construct a point of coordinates $(x_1-x_2,0)$.
\end{prob}
\begin{code}
    Consider the x-axis $r(0,0)$, and let $(x_1,y_1), (x_2,y_2)$ respectively be $P_i, P_j$. The subroutine return the index of the point $(x_1-x_2,0)$.
    
    	\begin{tabular} {p{0.1cm} p{7.5cm} p{6.5cm}} 
		\textbf{\texttt{diff(i,j)}} &	& signature of the subroutine\\
		\noalign{\vskip 1mm}
		& \texttt{onRod(0,0,[0])} & $P_{[0]}$ lies on the x-axis\\
		\noalign{\vskip 1mm}
		& \texttt{onRod(0,0,[1])} & $P_{[1]}$ lies on the x-axis\\
		\noalign{\vskip 1mm}
		& \texttt{onRod(perp(0,0,i),0,[1])} &  $P_{[1]}$ has coordinates $(x_1,0)$\\
		\noalign{\vskip 1mm}
		& \texttt{dist(sum([0],j),0,0,[1])} & the abscissa of $P_{[1]}$ has to be the sum of the abscissae of $P_{[0]}$ and $P_j$\\
		\noalign{\vskip 1mm}
		& \texttt{return([0])} & the sought point $P_{[0]}$ has coordinates $(x_1-x_2,0)$\\
		\noalign{\vskip 1mm}
	\end{tabular}
\end{code}
\begin{proof}
	Consider $P_{[0]}, P_{[1]}$ on the x-axis, constrain $P_{[1]}$ to have the same abscissa of $P_j$ (thus $P_{[1]}=(x_1,0)$). Using Problem \ref{sum}, we can constrain $P_{[1]}$ to have as abscissa the sum of the abscissae of $P_{[0]}$ and $P_i$: hence $P_{[0]}=(x_1-x_2,0)$.
\end{proof}

\begin{prob}[multiplication]\label{mult}
	Given two points of Cartesian coordinates $(x_1,y_1)$ and $(x_2,y_2)$, construct a point of coordinates $(x_1\cdot x_2,0)$.
\end{prob}
\begin{code}
    Let $(0,0),(1,0),(x_1,y_1), (x_2,y_2)$ respectively be $P_0, P_1, P_i, P_j$, and let the x-axis be $r(0,0)$. The subroutine return the index of the point $(x_1\cdot x_2,0)$.
    
    	\begin{tabular} {p{0.1cm} p{7cm} p{7cm}} 
		\textbf{\texttt{mult(i,j)}} &	& signature of the subroutine\\
		\noalign{\vskip 1mm}
		& \texttt{onRod(parall(0,0,inv(j)),0,[0])} & a new point $P_{[0]}$ lies on the rod parallel to the x-axis passing through $(y_2,x_2)$\\
		\noalign{\vskip 1mm}
		& \texttt{onRod(perp(0,0,1),0,[0])} & $P_{[0]}$ is constrained to have $1$ as abscissa\\
		\noalign{\vskip 1mm}
		& \texttt{onRod([0],0,0)} & the origin lies on the rod $r([0],0)$\\
		\noalign{\vskip 1mm}
		& \texttt{onRod(perp(0,0,i),0,[1])} & a new point $P_{[1]}$ lies on a rod perpendicular to the x-axis passing through $P_i$\\
		\noalign{\vskip 1mm}
		& \texttt{onRod([0],0,[1])} & the point $P_{[1]}$ is constrained to lie on the rod passing through the origin and $(1,x_2)$\\
		\noalign{\vskip 1mm}
		& \texttt{onRod(parall(0,0,[1]),0,[2])} & $P_{[2]}$ is constrained on the rod parallel to the x-axis passing through $(x_1,x_1\cdot x_2)$\\
		\noalign{\vskip 1mm}
		& \texttt{onRod(perp(0,0,0),0,[2])} & $P_{[2]}$ has coordinates $(0,x_1\cdot x_2)$\\
		\noalign{\vskip 1mm}
		& \texttt{return(inv([2]))} & the sought point is the one obtained inverting abscissa and ordinate of $P_{[2]}$ \\
		\noalign{\vskip 1mm}
	\end{tabular}
\end{code}
\begin{figure}
	\center \includegraphics[width=.55\textwidth] {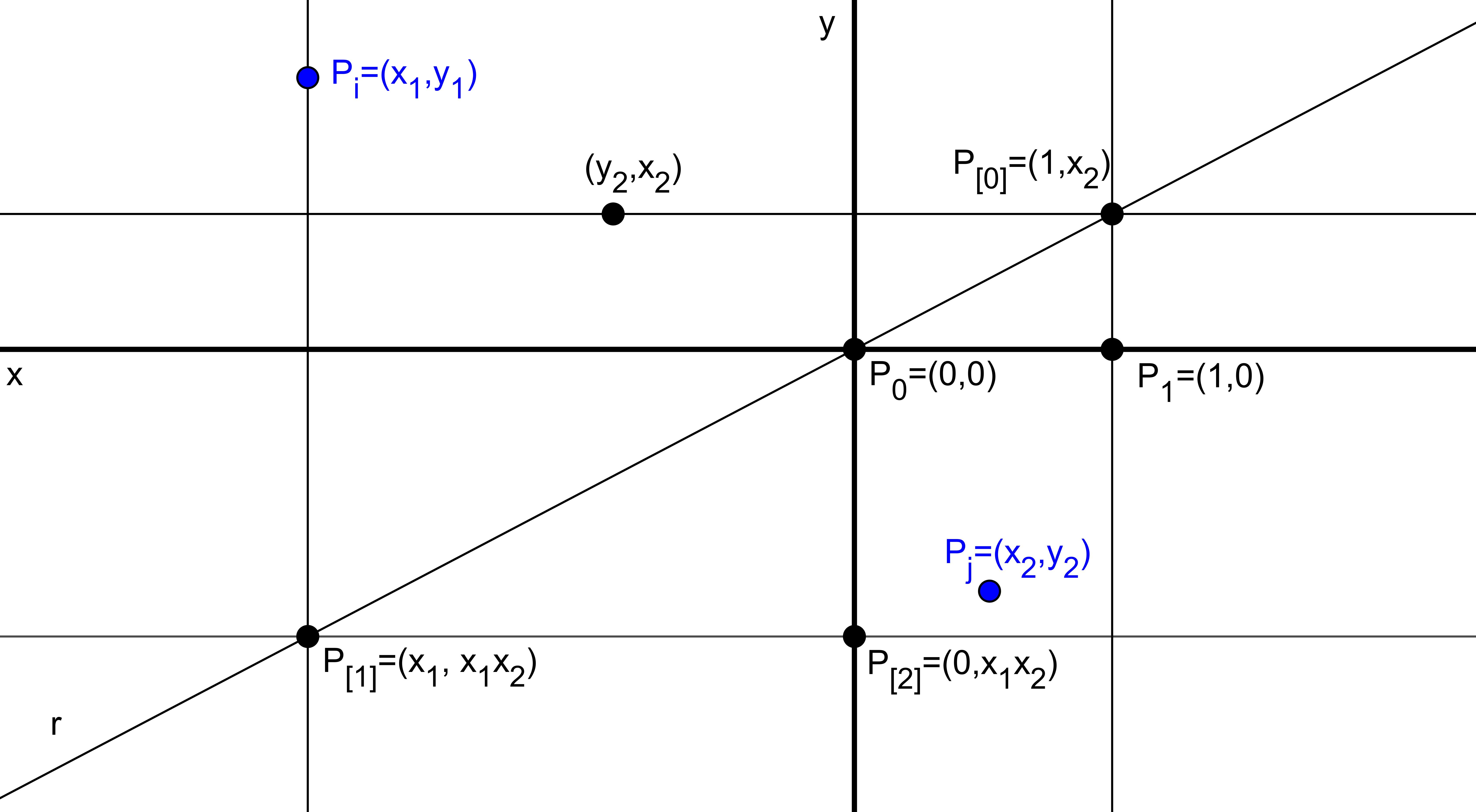}
	\caption[Construction for the multiplication]{Construction for the multiplication of the abscissae of $(x_1,y_1)$ and $(x_2,y_2)$. To obtain $(x_1\cdot x_2,0)$ we can invert the coordinates of $P_{[2]}=(0,x_1\cdot x_2).$ }\label{mult_rod}
\end{figure}
\begin{proof}
	As it happens in Descartes' interpretation of multiplication (the length $c=a\cdot b$ is given by the proportion $a:c = 1:b$), we need to use the unit length. 
	According to Problem \ref{inverse}, construct the point $(y_2,x_2)$. Considering the intersection of the rod parallel to the x-axis through $(y_2,x_2)$ and the rod parallel to the y-axis through $(1,0)$, we obtain the point of coordinates $(1,x_2)$.
	We can introduce the rod $r$ joined in $(1,x_2)$ passing through the origin $(0,0)$. As visible in Fig. \ref{mult_rod}, the intersection of $r$ with the rod parallel to the y-axis passing through $(x_1,y_1)$ determines the point $(x_1, x_1\cdot x_2)$. If we project it on the y-axis we obtain $(0, x_1\cdot x_2)$, that for Problem \ref{inverse} gives us the wanted $(x_1\cdot x_2,0)$.
\end{proof}

The possibility of performing addition and multiplication with algebraic machines is useful to analytically define the behaviour of algebraic machines and general TMMs in section \ref{manifolds}.

\section{Differential algebra (symbolic computation)}\label{differential_algebra}
In Cartesian geometry, polynomial algebra is used as finite tool for analysis. 
In the proposed differential extension, we substitute polynomials with differential polynomials, machines for algebraic constructions 
(algebraic machines) with TMMs, and algebraic curves with manifolds of zeros of differential polynomials. In this section, we delve deeper into the analytical counterpart of TMMs, the \textit{differential algebra}, specifically \textit{differential elimination}. The peculiarity of this approach is that it is algorithmically implementable (it is part of {computer algebra}): its finite symbolic manipulation does not need any reference to infinitary objects (as it happens in {infinitesimal calculus}).
These algebraic tools allow answering some questions about TMMs in section \ref{diff_problems}.

{Differential algebra} started with \cite{Rit1932}, where Ritt introduced suitable algebraic tools for differential equations. These results have been reformulated in more and more algebraic way in \cite{Rit1950} and later in \cite{Kol1973}.
Under both Ritt and Kolchin, basic differential algebra was developed from a constructive view point and the foundation they built has been advanced and extended to become applicable in symbolic computation, mainly thanks to the passage from old constructive methods (Ritt-Seidenberg algorithm of \cite{Sei1956}) to more recent computational complexity optimizations with Gr\"obner bases-like approach (firstly introduced in \cite{Car1989})\footnote
{Such ideas have been developed in Computer Algebra Systems, e.g. in the  package \emph{DifferentialAlgebra} (see \url{http://www.maplesoft.com/support/help/maple/view.aspx?path=DifferentialAlgebra}).
This package is based on the software $BLAD$ (standing for \emph{Biblioth\`eques Lilloises d'Alg\`ebre Diff\'erentielle}), developed in the \emph{C} programming language by F. Boulier. While \emph{DifferentialAlgebra} is a package for the commercial software \textit{Maple}, the $BLAD$ software is freely available online at \url{http://www.lifl.fr/~boulier/pmwiki/pmwiki.php?n=Main.BLAD}.
Furthermore, even though yet under construction, a differential-algebra experimental package for the open software \emph{SageMath} can be freely downloaded at the link \url{https://trac.sagemath.org/ticket/13268}. 
Another free alternative is \emph{ApCoCoA}, available at \url{apcocoa.org} (for our purposes, we have to cite the package \emph{diffalg}), a software package based on \emph{CoCoA}, \url{http://cocoa.dima.unige.it}. 
}. For a brief introduction to these computational problems and the relative historical evolution see \cite[pp. 110--111]{Bou2007}.
\\

The aim of differential algebra is to provide an algebraic theory for differential equations both ordinary or with partial derivatives. In particular, its tools and notations are an extension of commutative algebra. 
To give a short introduction to differential algebra, we recall \cite[pp. 112--116]{Bou2007} because of the clarity, the brevity, and the adherence with our aims (according to the kind of constraints obtained by TMMs, we are only interested in ordinary differential equations). 

Similarly to classic algebraic geometry, we consider \emph{differential polynomials}. 
In this case, out of the binary operations of sum and multiplication, we have to introduce the unitary operation of \emph{derivation}. The derivation must be distributive over addition (for every $a,b\in R$ it holds $D(a+b)=D(a)+D(b)$) and must obey the product rule (also called \textit{Leibniz rule},
$D(ab)=D(a)b + a D(b)$). Adopting the standard notation for ordinary derivatives, from now on we write $a'$ instead of $D(a)$.

For our purposes, the coefficients of such differential polynomials are rational numbers. Specifically, given a finite set $U$ of variables, named \emph{differential indeterminates}, a differential polynomial on $U$ is a polynomial on $U$ and the relative derivatives $\Theta U$ (if $U=\{x_1,x_2\}$, $\Theta U = \{x_1, x_1', x_1'',  \ldots, x_2, x_2', x_2'', \ldots \}$ and an example of differential polynomial is $\frac{1}{3} x_2'^2-5x_1^3x_1''^2x_2 + \frac{2}{13} x_1''^2 x_1'''x_2 x_2'^4$). In this paper, differential indeterminates can be considered as real functions depending on the single independent variable $t$, which we may think as the {time}. We also refer to differential indeterminates as \emph{dependent variables}. Considering by $\mathbb Q\{U\}$ the set of all the differential polynomials with rational coefficients on the variables $U$, it is a \textit{ring} (i.e. a mathematical structure equipped with sum and multiplication and satisfying certain properties: for an introduction to algebraic topics in the non-differential case see \cite{Lan2005}) with a derivation, thus a \textit{differential ring}.

The set of all the polynomials solving some polynomial conditions is captured by a structure named \textit{ideal}. 
Given an ideal $I$ of polynomials: $I$ contains the null polynomial; the sum of two polynomials in $I$ belongs to $I$; the multiplication of a polynomial in $I$ with another polynomial (not necessarily in $I$) still belongs to $I$.
In algebraic geometry, the set of polynomials satisfied by a given polynomial system forms an ideal that is also radical: an ideal $I$ is said to be \emph{radical} if $a\in I$ whenever there exists some $p\in \mathbb N$ so that $a^p\in I$. 
%
In the differential case, an ideal $ I$ is a \emph{differential ideal} if it is stable under derivation, which is $a'\in I$, for all $a\in I$. Besides, exactly as in non-differential case, a differential ideal $I$ is \emph{radical} if $a^p\in I$ implies $a\in I$ for any integer $p>0$.

The set of all the ``{differential and algebraic consequences}'' of the differential polynomials in a system $\Sigma$ is the radical differential ideal generated by $\Sigma$. 
{To observe how radical differential ideal are related to the study of the solution of a system of differential polynomials, 
	consider $x'^2-4x=0$ ($x$ is the only dependent variable). The analytical solutions are the zero function $x(t)=0$ and the family of parabolas $x(t)=(t+c)^2$ where $c$ is an arbitrary constant. These are also solutions of all the derivatives of $x'^2-4x=0$ (i.e. {$2x'(x''-2)=0, 2x'x''+2x''(x''-2)=0, \ldots$}) and of 
	every differential polynomial a power of which is a finite linear combination of the derivatives with arbitrary differential polynomials as coefficients, i.e. every element of the radical of the differential ideal generated by $x'^2-4x$. 
}

When we are interested in a restriction of all the variables, we take a \textit{projection} of an ideal, and the operation is called \textit{elimination}. The projection of an ideal is still an ideal.
In the elimination process for both purely algebraic and differential systems, we need as input a system of (differential) polynomials and an order defining the priority of the variables to be eliminated, so called \textit{ranking}. The output is a system (or a family of systems when splitting is necessary) that is equivalent to the input system restricted on some variables. 
Even if in practice the worst case complexity of the algorithms makes problems untreatable, in principle elimination is always possible.
\\

For the differential elimination (and in general to decide the membership of a differential polynomial in a radical differential ideal) the key algorithm is \emph{Rosenfeld-Gr\"obner} one. As readable in the description of the command in Maple\footnote
{Cf. \url{https://www.maplesoft.com/support/help/Maple/view.aspx?path=DifferentialAlgebra/RosenfeldGroebner}}, given a system $\Sigma$ containing differential-polynomial equations and inequations, the algorithm splits the given system into other systems defined by certain equations and inequations (the result and the number of cases depend on the ranking of the variables). 
The solutions of $\Sigma$ are given by the union of the general solutions of each of the returned systems.
Every system is an algebraic structure named \emph{differential regular chain},
and  the radical differential ideal generated by $\Sigma$ is the intersection of all the obtained differential regular chains.
\\

About the differential ranking, 
if $U$ is a finite set of dependent variables, a \emph{ranking} over $U$ is a total ordering over the set $\Theta U$ of all the derivatives of the elements of $U$ which satisfies, for all $a,b\in\Theta U$,
$a'>a$ and $a>b\Rightarrow a'>b'$. 
When $U=\{a\}$ (there is a unique dependent variable), there exists only one ranking: $\cdots > a''>a'>a$.
The choice of the ranking is non-trivial when we have more dependent variables. For our purposes, we have to introduce the \textit{orderly} and \textit{eliminating} rankings.

A ranking is said to be \emph{orderly} if, for every $a,b\in U$ and for every positive integer value of $i$ and $j$, $i>j\Rightarrow a^{(i)}>b^{(j)}$. 

If $U$ and $V$ are two finite sets of differential variables, one denotes $U \gg V$ every ranking so that any derivative of any element of $U$ is greater than any derivative of any element of $V$. Such rankings are said to \emph{eliminate} $U$ with respect to $V$. 

Fixed a ranking, the \emph{leader} is the highest ranking derivative appearing in a differential polynomial.
Thus, given $\frac{1}{3} x_2'^2-5x_1^3x_1''^2x_2 + \frac{2}{13} x_1''^2 x_1'''x_2 x_2'^4$, with any orderly ranking the leader is $x_1'''$ (there are no $x_2$ with derivative more than 1). We have the same leader with the ranking eliminating $x_1$. On the contrary, with the ranking eliminating $x_2$ the leader is $x_2'$.
\\

To sum up, recalling \cite[pp. 41--42]{Hub2003}, given a system of differential polynomials $\Sigma$ in the dependent variables $x_1,\ldots,x_n$, 
with an appropriate choice of the ranking we can:
\begin{itemize}
	\item check whether a differential polynomial is a solution of $\Sigma=0$;
	
	\item find the differential polynomials satisfied by the solutions of $\Sigma=0$ in a subset of the dependent variables (we can obtain the equations governing the behavior of the components $x_1, \ldots, x_m$, with $m<n$);
	
	\item find the lower order differential polynomials satisfied by the solutions of $\Sigma=0$ (in particular, we can inquire whether the solutions of the system are	constrained by purely algebraic equations).
\end{itemize}

Even though very powerful, the introduced methods do not provide the answer to all the interesting questions of differential algebra. 
With regard to initial value problems from a computational symbolic perspective,\footnote
{For example, we are interested in the following problem: given two TMMs with their relative initial configurations, are their behaviors equivalent? Analytically, the question arises: given two systems of differential equations with the relative initial conditions, are the systems equivalent? We are looking for an algorithm to symbolically solve this problem. Differential algebra language does not permit even to express this problem because we need to explicitly state the relation between the dependent variables and the independent one (to pose the initial condition).
\label{note:initial_values}}
a lot left to do. Even though \cite{PS2007} and the approach proposed by Markus Rosenkranz with
regard to symbolic methods for (linear) boundary problems (e.g. \cite{RRT2012}), at my knowledge the symbolic solution of general initial value problems is far away from being solved.

\section{Geometry (machines behavior)}\label{manifolds}
To describe the behavior defined by TMMs, we adopt the \textit{behavioral approach} of mathematical models \cite[pp. 1--8]{PW1998}. The main difference between the behavioral approach and the input/output one is that in the first one we consider all the variables without the need of distinguishing them between input and output. The advantage of missing this distinction comes from the fact that considering interconnection between components (the so-called \textit{feedback}), it is generally hard or impossible to understand which variables are inputs and which ones are outputs. TMMs were firstly introduced adopting the input/output approach \cite{Mil2012}, while in this paper we use the behavioral approach to analytically study the machines with differential algebra instead of classical infinitesimal calculus.

A mathematical model posits that some things can happen, while others cannot.
We can formalize this idea by stating that a mathematical model selects a certain subset from a \textit{universum} of possibilities. This subset consists of occurrences that the model allows, that it declares possible. We can refer to the subset in question as the \textit{behavior} of the mathematical model.
Such exclusion laws are usually expressed in terms of equations in some variables. 
The behavior obtained considering all the variables is called \textit{total}. If we want to restrict only to the some variables, we speak of \textit{restricted behavior} (or simply behavior, restricted may be implicit). With this approach two machines are equivalent if they have the same behavior.
Before deepening the exploration of TMMs, we begin to explore the behavior of algebraic machines.

\begin{prop}\label{prop:total_algebraic}
The total behavior of algebraic machines with $n$ points and $m$ rods is a real algebraic set with integer coefficients in ${2(n+m)}$ variables.
\end{prop}
\begin{proof}
    First of all, for algebraic machines we can consider as behavior the set of the configurations allowed by the constraints of the machine, i.e. the possible contemporary positions of the various points. Being on a plane, each of the two coordinates of any point has to be a real value.
    Given the possibility of translating \texttt{onRod} and \texttt{dist} in algebraic polynomials with integer coefficients (cf. section \ref{language}), we can consider as variables the coordinates of the $n$ points $P_i$. At them we have to add, for each of the $m$ rods, the point $Q_{j,k}$. Therefore, in the universum $\mathbb R^{2(n+m)}$, the behavior (the possible configurations) is the solution of a system of polynomials with integer coefficients on $2(n+m)$ variables.
\end{proof}
Note that, from an algebraic perspective, it is better to consider rational coefficients instead of integer ones to allow the existence of the inverse of every non-zero coefficient (to define a so called \textit{polynomial ring}). However, multiplying by the  minimum common multiple of all the denominators, every polynomial with rational coefficients is equivalent to another one with only  integer coefficients, hence we continue considering only integer coefficients.
\begin{prop}\label{prop:algebraic}
Given any polynomial $p$ with integer coefficients in $n$ real variables, we can consider an algebraic machine having as restricted behavior exactly the zero set of $p$.
\begin{proof}
    Consider $p$ on $x_1,\ldots,x_n$ with integer coefficients, let such coefficients be $c_1, \ldots, c_k$. Introduce the $n$ moving points $P_i=(x_i,0)$  by \texttt{onRod(0,0,i)} (keep in mind that the x-axis is $r(0,0)$) and the $k$ fixed points $P_{n+j}=(c_j,0)$ by \texttt{dist(0,0,$c_j$,n+j)}. The polynomial $p$ is made up by sums and multiplications of $x_i$ and $c_j$. Thus, using the code of Problems \ref{sum} and \ref{mult}, we can construct the point $P_h=(p,0)$ (with a certain index $h$) in function of $x_1,\ldots, x_n$. To conclude, we can impose $p$ to be constantly equal to 0 by \texttt{dist(0,0,0,h)}.
    Restricting the behavior to the abscissae of $P_1, \ldots, P_n$, the proposed machine provides exactly the sought solution of the polynomial $p$.
\end{proof}
\end{prop}

Note that, once physically posed the constraints, not every configuration is always reachable given a certain initial condition. That happens because real algebraic sets (i.e. the solution of systems of algebraic polynomials on real variables) are not always made up by connected components (consider the hyperbola $xy-1=0$: it is made up by two unconnected branches). However, every connected branch of an algebraic set can be continuously traveled by a machine given suitable initial conditions.


Now we have to manage the passage from total to restricted behavior. The projection of any real algebraic set (i.e. the set obtained eliminating some of the original real variables) is named semi-algebraic set (and every real semi-algebraic one can be obtained as the projection of a real algebraic set) \cite{Bas2006}. 
Any semi-algebraic set can be represented as finite union of sets each one defined by some polynomial equations and inequalities.

Projection of real algebraic sets can be performed by computer algebra software like Maple (that can be used also to analyse the behavior of general TMMs)\footnote
{To perform such projection one can use the \textit{RegularChains} package \url{https://www.maplesoft.com/support/help/maple/view.aspx?path=RegularChains} (out of clarifying how to use the package, a mathematical definition of regular chains is provided), with also the subpackage \textit{SemiAlgebraicSetTools} for the function \textit{Projection} \url{https://www.maplesoft.com/support/help/maple/view.aspx?path=RegularChains\%2fConstructibleSetTools\%2fProjection}.
To clarify the ideas, it follows the code to get the behavior of the point $P_4$ in the example \ref{linkages} at page \pageref{linkages} (see Fig. \ref{Watt} for the machine).

	\begin{tabular} {p{8cm} p{5.5cm}} 
		\texttt{with(RegularChains): with(SemiAlgebraicSetTools)}	& load the packages\\
		\noalign{\vskip 1mm}
		\texttt{R:= PolynomialRing([x2, y2, x3, y3, x4, y4])} & define the variables of our polynomials with a given order\\ 
		\noalign{\vskip 1mm}
		\begin{tabular}{l}
		     \texttt{eq1:= x2\^{}2+y2\^{}2=1}\\ \texttt{eq2:= (x3-1)\^{}2+y3\^{}2=1}\\ \texttt{eq3:= x2-2*x4+x3=0}\\ \texttt{eq4:= y2-2*y4+y3=0}\\ \texttt{eq5:= (x2-x3)\^{}2+(y2-y3)\^{}2=4}
		 \end{tabular} & define the various equations\\
		\noalign{\vskip 1mm}
		\texttt{proj:= Projection([eq1,eq2,eq3,eq4,eq5],2,R)} & compute the projection of the equations on the last two variables considering the given order
		\\ 
		\noalign{\vskip 1mm}
		\texttt{Display(proj,R)}     & show the various components of the projection
		\\ 
		\noalign{\vskip 1mm}
	\end{tabular}
	
Thus, non considering the imaginary results and merging the singular solutions in the general one, we have that the locus of the point $P_4$ is given by the equation $4x^6-12x^5+(13+12y^2)x^4+(-6-24y^2)x^3+(18y^2+12y^4+1)x^2+(-12y^4-6y^2)x-3y^2+5y^4+4y^6 = 0$. 

Neglecting practical computer limitation, eliminations are always theoretically computable. For example one could prove by computer algebra that the algebraic machine introduced for the sum in problem \ref{sum} works properly without any geometrical consideration. One should translate all the instructions in algebraic equations and then consider the projection on three variables: the two addends (the abscissae of $P_i$ and $P_j$, i.e. $x_1$ and $x_2$) and the variable that should give the result of the operation (the abscissa of $P_{[1]}$, denote it by $t$). The projection has to provide an equation equivalent to $x_1+x_2-t=0$. Obviously, similar reasonings are obtainable for the operations of difference and multiplication, and, suitably interpreting the geometrical properties in analytic terms, also for all the other problems solved in the section \ref{Algebraic_machines}.
\label{note:realAlg}}.

We complete the section with the characterization of the behavior of algebraic machines.
\begin{thm}[Algebraic universality]\label{alg_universality}
    Considering $n$ variables, the behavior of algebraic machines coincides with any semi-algebraic set with integer coefficients.
\end{thm}
\begin{proof}
    We split the theorem in two: 1. any machine behavior is a semi-algebraic set; 2. for any semi-algebraic set, we can consider a machine with this behavior.
    
    1. By Prop. \ref{prop:total_algebraic}, the total behavior is a real algebraic set. Thus, restricting to $n$ variables, the behavior has to be a semi-algebraic set.
    
    2. Any semi-algebraic set can be represented as finite union of sets each one defined by some polynomial equations and inequalities on $n$ variables. 
    For every inequality $q_i(x_1,\ldots,x_n)>0$, we can introduce a new variable $t_i$ and the polynomial $\tilde{q_i}(x_1,\ldots,x_n,t_i)=q_i\cdot t_i^2-1$. We can note that, posing $\tilde{q_i}=0$, we have that $q_i\neq 0$ ($q_i$ has to divide $1$) and $q_i\geq 0$ ($q_i$ multiplied by a non-negative value must give a positive value): that means that, thanks to the introduction of new variables, we can obtain the desired inequalities as projection using the new polynomials $\tilde{q_i}$.
    
    We can also consider that the system of real polynomials $p_1=\ldots=p_l=0$ is equivalent to a single polynomial $(p_1)^2+\ldots+(p_l)^2=0$. Hence, a semi-algebraic set can be considered as the projection of a finite union of sets each one defined by a polynomial in $x_1,\ldots,x_n,t_1,\ldots,t_m$: let such polynomials be $P_1,\ldots,P_k$. Finally, the union of such sets has to satisfy the polynomial $P= P_1\cdot \ldots \cdot P_k$, and the semi-algebraic set has to be the projection of $P=0$ on the variables $x_1,\ldots, x_n$. By Prop. \ref{prop:algebraic}, the zero-set of $P$ can be constructed by an algebraic machines. Thus, restricting on the variables $x_1,\ldots, x_n$, we got a machine with the sought behavior.  
\end{proof}


\subsection{Universum of TMMs}\label{almost_cycloid}
Differently from the case of algebraic machines, whose behavior is a subset of $\mathbb R^n$, the introduction of the wheel makes the previous representation ineffective. 
To evince it, we propose the example of a TMM: the set of its reachable points is a real semi-algebraic set, so it can be obtained with an algebraic machine, but we will intuitively see why its behavior is substantially different from the ones of algebraic machines.

\begin{ex}\label{example:cycloid}
Given $P_2 = (t, 0)$ moving on the abscissa and $P_{[0]} = (t, 1)$, consider $P_{[1]}$ s.t. $\overline{P_2P_{[1]}}=1$. In $P_{[1]}$ we can place a wheel so that the tangent to the curve traced by $P_{[1]}$ is always in direction of $P_{[0]}$ (see Fig. \ref{cycloid}). 
\begin{code}
    First example of a TMM that is not an algebraic machine. Consider $P_0, P_1, r(0,0)$ as usual. 
    
    	\begin{tabular} {p{0.1cm} p{7.5cm} p{6.5cm}} 
		& \texttt{onRod(0,0,2)} & consider $P_2$ on $r(0,0)$\\
		\noalign{\vskip 1mm}
		& \texttt{dist(perp(0,0,2),0,1,[0])} & $P_{[0]}$ is constrained to stay on the perpendicular to $r(0,0)$ passing through $P_2$, one unit away from  $P_2$\\
		\noalign{\vskip 1mm}
		& \texttt{dist(2,0,1,[1])} & a new point $P_{[1]}$ is put at distance 1 from $P_2$\\
		\noalign{\vskip 1mm}
		& \texttt{onRod([1],0,[0])} & the point $P_{[0]}$ is constrained to lie on the rod $r([1],0)$\\
		\noalign{\vskip 1mm}
		& \texttt{wheel([1],0)} & pose a wheel in $r([1],0)$ at $P_{[1]}$\\
		\noalign{\vskip 1mm}
	\end{tabular}
\end{code}
\begin{figure}
	\center \includegraphics[scale=.2] {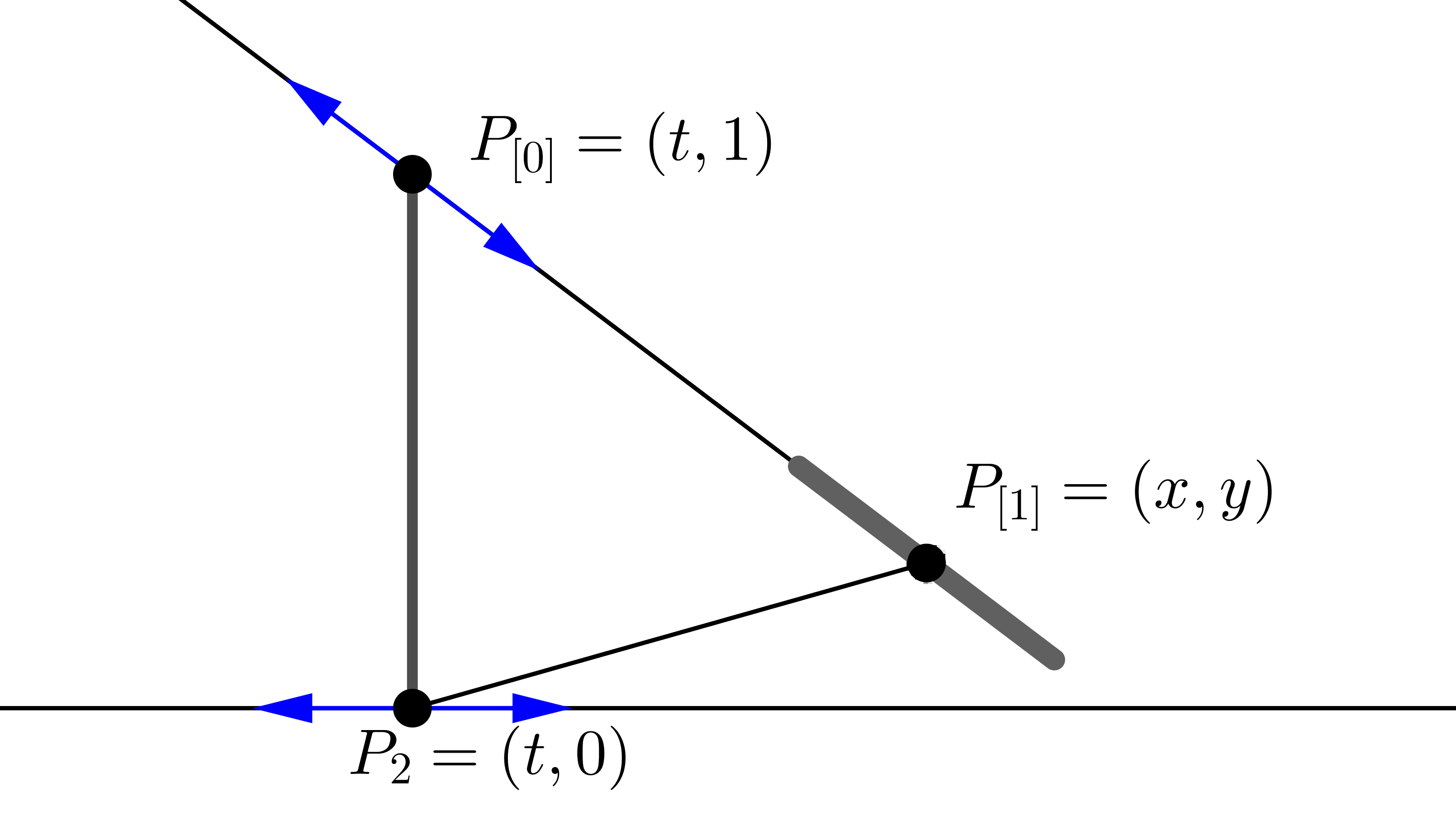}
	\caption[A simple TMM]{A simple TMM (the point $P_2$ moves along a line, $P_{[1]}$ rotates around).
	}\label{cycloid}
\end{figure}
\begin{figure}
	\center \includegraphics[scale=.9] {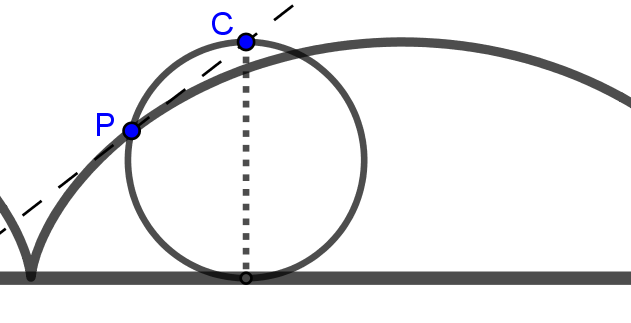}
	\caption[A property of the cycloid]{A cycloid can be traced by a point $P$ fixed on a rolling disk. Calling $C$ the top of the disk, the tangent at $P$ has to pass through $C$. }\label{cycloid-dim}
\end{figure}
\end{ex}
Note that, although by construction $P_{[0]}=(t,\pm 1)$, to avoid useless complications we consider only $P_{[0]}=(t,1)$.
While one moves $P_2$ along the abscissa, if the ordinate of $P_{[1]}$ is strictly less than 1, it has to describe an arc of cycloid, because of the geometrical property shown in Fig. \ref{cycloid-dim} (for a precise analytic proof see the section \ref{analysis_cycloid}, page \pageref{analysis_cycloid}). Remind that the cycloid is a transcendental curve, thus it cannot be traced with 1 degree of freedom by algebraic machines.
On the contrary, when $P_{[1]}$ assumes coordinates $(t, 1)$, it can move not only along a cycloid but also in a purely horizontal way, losing the uniqueness (a similar machine was introduced in \cite[pp. 10-12]{Mil2015}).

\begin{figure}
	\center \includegraphics[scale=.9] {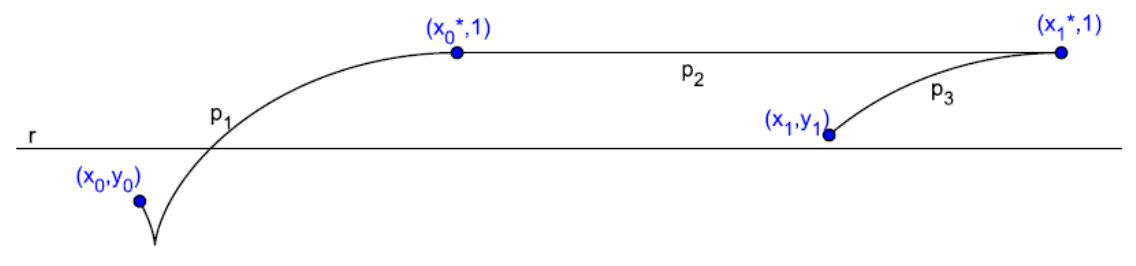}
	\caption[Path connecting two points with the simple TMM]{For any two points $(x_0,y_0)$ and $(x_1,y_1)$ in the strip $]-\infty,+\infty[\times[-1,1]$ there is a path (the combination of the paths $p_1, p_2, p_3$) satisfying the constraints of the machine seen in Fig. \ref{cycloid}.}\label{cycloid-path}
\end{figure}
It means that, given any initial position $(x_0,y_0)$ of $P_{[1]}$ in the strip $\left ]-\infty, +\infty\right [ \times \left [-1,1 \right ]$, any other value $(x_1,y_1)$ in the strip can be reached by $P_{[1]}$:
call $(x_0^*,1)$ and $(x_1^*,1)$ the first apex (going from left to right) of the cycloid starting respectively in $(x_0,y_0)$ and $(x_1,y_1)$. As visible in Fig. \ref{cycloid-path}, $P_{[1]}$ can reach $(x_1,y_1)$ from $(x_0,y_0)$ decomposing the motion in three parts: first, $P_{[1]}$ reaches $(x_0^*,1)$ (it is possible because they are on the same branch of cycloid); second, $P_{[1]}$ reaches $(x_1^*,1)$ (they are on the horizontal line $y=1$); third, $P_{[1]}$ reaches $(x_1,y_1)$ (because they are on the same branch of cycloid).

So, restricting the behavior to the coordinates $(x,y)$ of $P_{[1]}$, the space of the reachable configurations is exactly the strip $\left ]-\infty,+\infty\right [\times \left [-1,1 \right ]$.
This strip is a real semi-algebraic set, so it can be considered as the behavior of an algebraic machine (e.g. our TMM without the last wheel instruction). For the behavioral approach, two systems/machines are equivalent if they have the same behavior. If we consider the set of the reachable configurations as behavior of TMMs, the proposed machine is equivalent to an algebraic one.
However, for any algebraic machine, any path internal to the space of the reachable configuration is a path that can be walked by the machine. On the contrary, for this TMM, $P_{[1]}$ can walk only certain trajectories. Thus, we cannot consider a subset of $\mathbb R^n$ as universum for TMMs; we need something else. 
In particular, considering this example, the universum of a TMM can be made up by a (generally infinite) set of curves satisfying both the configuration conditions of the holonomic constraints and the path conditions imposed by non-holonomic ones. Let's define it more precisely.
\\

Differently from the synthetic approach, differential geometry introduces calculus for the investigation of curves. In particular, 
curves are represented in a parametrized form as a class of equivalence on vector-valued functions\footnote
{A \textit{parametric curve} $\gamma$ is a vector-valued function $I \to {\mathbb R}^n$ (where $I$ is a non-empty interval of real numbers) of class $C^r$ (i.e. $\gamma$ is $r$ times continuously differentiable, eventually also $\infty$ times differentiable).
	If we consider $t\in I$, $t$ is the parameter of $\gamma$, and $\gamma(I)$ is the image of the curve (considering $t$ as time, $\gamma(t)$ represents the trajectory of a moving particle).
	
	The same image $\gamma(I)$ can be described by several different $C^r$ parametric curves: the aim of differential geometry is to study the curve independently from \textit{reparametrizations}. In doing that, we can consider curves as an equivalence class on the set of parametric curves.
	The equivalence class is called a $C^r$ curve (equivalent $C^r$ curves have the same image). For a detailed discussion, see, for example, \cite{DoC1976}.
}.
Coming back to TMMs, we can continue the interpretation of variables as coordinates of specific points of machines as done in algebraic ones. But, unlike before, it is no longer enough to consider variable as real numbers, but, to introduce path constraints, we can consider these variables as real functions ($\mathbb R\to\mathbb R$), where the parameter can be considered as the time.  
Being an idealization of physical machines, we consider these functions to be $C^\infty$, i.e. \textit{smooth} functions.

With reference to the example of the machine in Fig. \ref{cycloid}, we need to consider as universum something like manifolds of curves. However, curves can be defined as classes of equivalence over vector-valued functions. So, to mathematically simplify the definition, we suggest to consider a ``{manifold of $C^\infty$ functions}'' as universum for TMMs. In particular, considering $n$ variables, these functions have to be $\mathbb R\to\mathbb R^n$.

Algebraic machines are a restriction of TMMs, so we can observe how the interpretation of the universum/behavior as real semi-algebraic set is reformulated as manifold of functions. From the point of view of paths, algebraic machines allow any path moving inside the defined semi-algebraic set $S\subset\mathbb R^n$, so the manifold of functions has to be made up by all the functions of class $C^\infty$ having their image inside $S$.

\subsection{Total behavior as solution of differential polynomial systems}\label{full_behavior_diff_machines}
We have just defined a manifold of smooth functions as universum of a TMM.
Variables are coordinates of specific points, and are considered as functions. 
%
As introduced in section \ref{language}, both wheel constraints and algebraic conditions are translatable in polynomials in the variables and their derivatives: as seen in the algebraic part (section \ref{differential_algebra}), such polynomials are named \textit{differential polynomials} and constitute the basis for differential algebra. More formally:
\begin{prop}\label{TMM-total_behavior}
    The total behavior of a TMM with $n$ points and $m$ rods is the manifold of all the smooth real functions $\mathbb R\to\mathbb R^{2(n+m)}$ satisfying a system $\Sigma$ of differential polynomial equations with integer coefficients.
\begin{proof}
    We just have to translate all the instructions defining the machine in differential equations (purely algebraic equations if the instruction is not \texttt{wheel}), taking care of introducing the auxiliary points $Q_{i,j}$ when introducing a rod. By construction, all the coefficients has to be integers.
    Renaming $x_1,\ldots,x_{2(n+m)}$ all the coordinates in function of the time $t$, and considering $p_1, \ldots, p_l$ the obtained differential polynomials on $x_1,\ldots,x_{2(n+m)}$, the total behavior is $\{(x_1,\ldots,x_{2(n+m)}) | x_i:\mathbb R \rightarrow \mathbb R, x_i\in C^\infty, p_1=\ldots =p_l=0\}$.
\end{proof}
\end{prop}

Note that we found an analytical form only to the total behavior: for the restricted behavior, in general we have to \textit{eliminate} the unwanted variables. 

\begin{prop}\label{solving_diff_systems}
    Given a system $\Sigma$ of differential polynomial equations with integer coefficients, we can construct a machine having as restricted behavior the manifold of the solutions of $\Sigma$. 
\begin{proof}
    First of all, we can convert $\Sigma$ in an equivalent system involving more variables but with only first derivatives. Let $y_1, \ldots, y_m$ be the variables of $\Sigma$, and let $k_i$ be the maximum derivative of $y_i$ present in $\Sigma$ (i.e. $y_i^{(k_i)}$ appears, but $y_i^{(k_i+1)}$ doesn't): for every variable we introduce new auxiliary variables $y_{i,j}$ (for $j=1,\ldots,k_i$) and the differential polynomials 
    \begin{equation}\label{eq:first_order}
        y_{i,0}-y_i\:,\: y_{i,1}-y_{i,0}'\:,\; \ldots\;,\: y_{i,k_i}-y_{i,k_i-1}'.
    \end{equation}
    Adding such new polynomials and modifying the system by substituting $y_i^k$ with $y_{i,k}$, we get an equivalent system involving only purely algebraic polynomials and the first order derivatives of (\ref{eq:first_order}). To simplify the notation and use a single index (still involving only first order derivatives), denote the various variables $y_{i,k}$ (for all $i$ and $k$) by $x_1,\ldots,x_n$.
    
    Working on real values, we can convert the system $\Sigma$ given by  $p_1=\ldots=p_l=0$ in a single differential polynomial  $p=(p_1)^2+\ldots+(p_l)^2=0$. 
    This differential polynomial is a polynomial on $x_1,\ldots,x_n$ and their first derivatives. As seen in Prop. \ref{prop:algebraic}, we can solve by algebraic machines any polynomial, hence we just have to show how to construct the derivatives of the variables $x_1,\ldots,x_n$ (this construction was firstly expressed in \cite{Mil2012} to solve polynomial Cauchy problems). For this purpose we introduce the following code.

\begin{figure}
	\center \includegraphics[scale=.3] {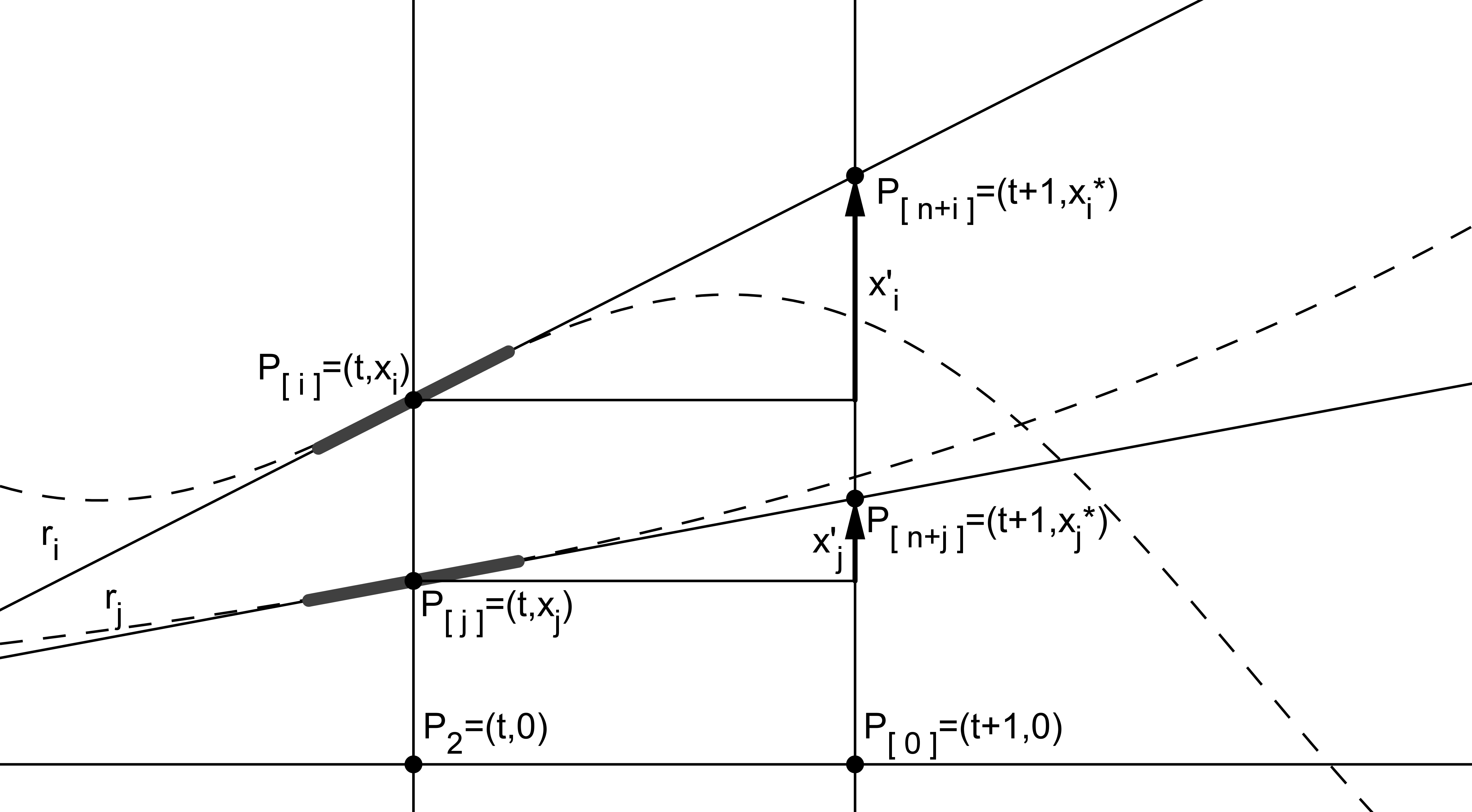}
	\caption[Construction of the derivative of the variables $x_i,x_j$]{Construction of the derivative of the variables $x_i,x_j$.}\label{diff_poly}
\end{figure}
\begin{code}
    Given the fixed points $P_0, P_1$ and $r(0,0)$ as abscissa, in this code we introduce the following points: $P_2=(t,0), P_{[0]}=(t+1,0), P_{[i]}=(t,x_i), P_{[i+n]}=(t+1,x_{i}+x'_{i})$ (for $i=1,\ldots,n$).
     
    	\begin{tabular} {p{0.1cm} p{7.7cm} p{6.5cm}} 
		& \texttt{onRod(0,0,2)} & consider $P_2$ on $r(0,0)$\\
		\noalign{\vskip 1mm}
		& \texttt{dist([0],0,0,sum(1,2))} & denoting $P_2=(t,0)$, $P_{[0]}$ is constrained to stay in $(t+1,0)$ \\
		\noalign{\vskip 1mm}
		&
		\begin{tabular}{l}
		     \texttt{onRod(perp(0,0,2),0,[1])} \\
		     \ldots\\
		     \texttt{onRod(perp(0,0,2),0,[n])}
		\end{tabular}
		 & each $P_{[i]}$ is constrained to lie on the rod perpendicular to $r(0,0)$ passing through $P_2$\\
		\noalign{\vskip 1mm}
		& 
		\begin{tabular}{l}
		     \texttt{onRod(perp(0,0,[0]),0,[n+1])} \\
		     \ldots\\
		     \texttt{onRod(perp(0,0,[0]),0,[2n])}\\
		\end{tabular}
		& each $P_{[n+i]}$ has to lie on the rod perpendicular to $r(0,0)$ passing through $P_{[0]}$\\
		\noalign{\vskip 1mm}
        & 
		\begin{tabular}{l}
		     \texttt{onRod([1],0,[n+1])} \\
		     \ldots\\
		     \texttt{onRod([n],0,[2n])}
		\end{tabular}
        & each rod $r([i],0)$ has to pass through $P_{[n+i]}$\\
		\noalign{\vskip 1mm}
        & 
		\begin{tabular}{l}
		     \texttt{wheel([1],0)} \\
		     \ldots\\
		     \texttt{wheel([n],0)}
		\end{tabular}
        & in these lines we introduce wheels in $r([i],0)$\\
		\noalign{\vskip 1mm}
	\end{tabular}
\end{code}

As Fig. \ref{diff_poly} illustrates, consider the point $P_2=(t,0)$ by a cart on the abscissa ($t$ can assume any real value). Note that $t$ is arbitrary, the important thing is that all the various $x_i$ are considered in correspondence of the same $t$: $t$ can be viewed as the \textit{independent variable} in function of which the various functions (dependent variables) are computed.
Then, consider the points $P_{[1]}=(t,x_1), \ldots , P_{[n]}=(t,x_n)$. On these points, we can put $n$ rods: call $r_i=r([i],0)$ the rod joined in $P_{[i]}$. Put also a wheel on every $r_i$ in correspondence of $P_{[i]}$. 
We can construct the rod of equation $x=t+1$: call $x_i^*$ the ordinate of the point $P_{[n+i]}$ in the intersection of $x=t+1$ and $r_i$.

For what has been observed about the role of the wheel, $r_i$ has to be tangent to the graph of $(t,x_i)$, hence $x_i^*$ will be $x_i+x'_i$. Obviously, it was not strictly necessary to construct the rod of equation $x=t+1$: in the case of a rod of equation $x=t+a$ (for any constant $a\neq 0$), the intersection of $r_i$ with the new rod is $(t+a,x_i+a x'_i)$ (in other words, $x_i^*=x_i + a x'_i$). 
It means that we can construct the point $(x'_i,0)$ that can be used as a new variable, and so the differential polynomial can be considered as a purely algebraic polynomial on $x_1,\ldots,x_n, x_1',\ldots,x_n'$. 
Specifically, the points $(x'_{i},0)$ are obtainable by \texttt{diff(inv([n+i]),inv([i]))} (with $i=1,\ldots,n$).

So, the possibility of solving polynomials with algebraic machines assures that, for every system of differential polynomial equations $\Sigma$, we can consider a TMM having as restricted behavior (restricted to the original $y_1,\ldots,y_m$) the solution of $\Sigma$.
\end{proof}
\end{prop}

\subsection{Note on ``{independentization}''}\label{independentization}
As a first example of passage from differential equation to TMM, we can consider the problem $y'=y$. To construct a machine solving it we can start considering a cart $(t,0)$ on a fixed rod (that we consider as abscissa), a rod perpendicular to the abscissa and translating according to the value of $t$, and on this rod the point $(t,y)$.
As already observed, instead of the rod of equation $x=t+1$, we can consider any other form $x=t+a$. In particular, it is simpler if we adopt $a=-1$. Thus, $y^*=y+ay'=0$ (for the problem is $y'=y$).
Therefore, we have to introduce the rod $r$ passing through $(t,y)$ and $(t-1,0)$, and to put a wheel on it in correspondence of $(t,y)$, obtaining the machine of Fig. \ref{exp_machine}.
Conceptually, this machine is constructively using the property of the exponential curve of having a fixed-length subtangent (i.e. the segment connecting $(t-1,0)$ and $(t,0)$).
\begin{figure}
	\center
	\includegraphics[scale=.5] {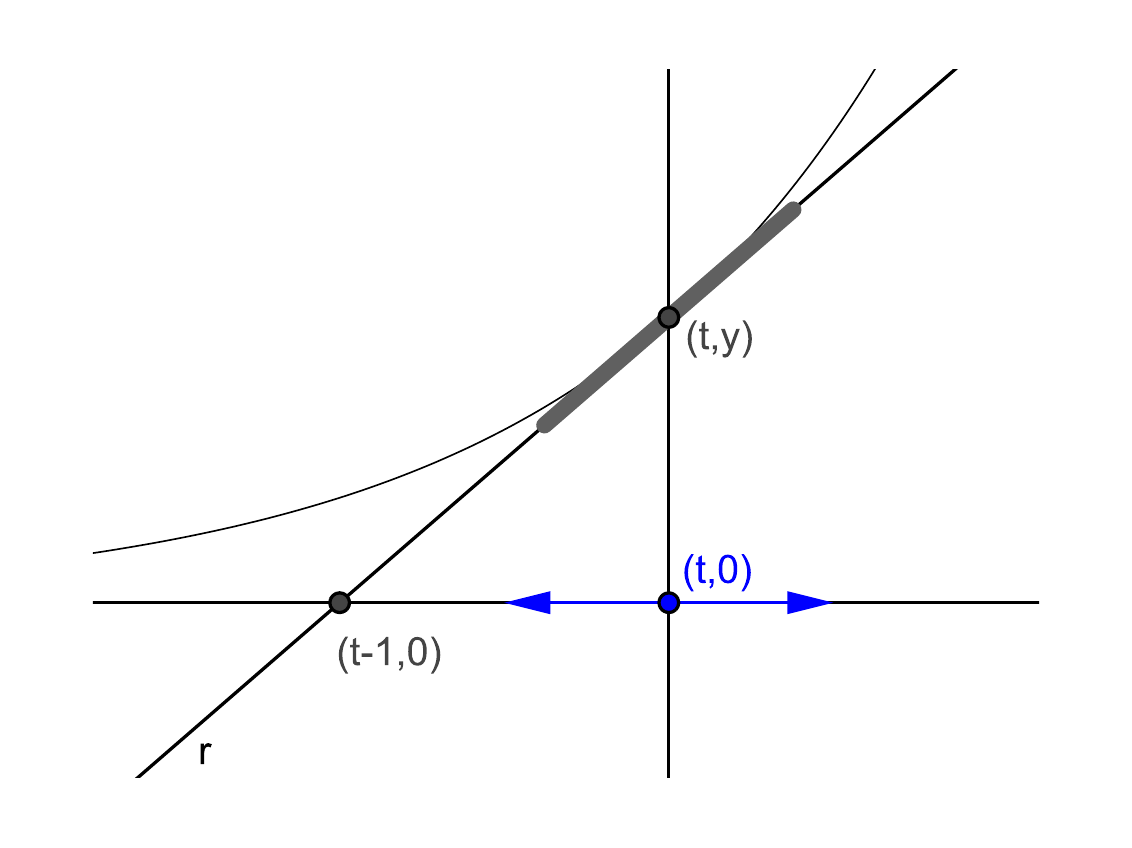}
	\includegraphics[scale=.3] {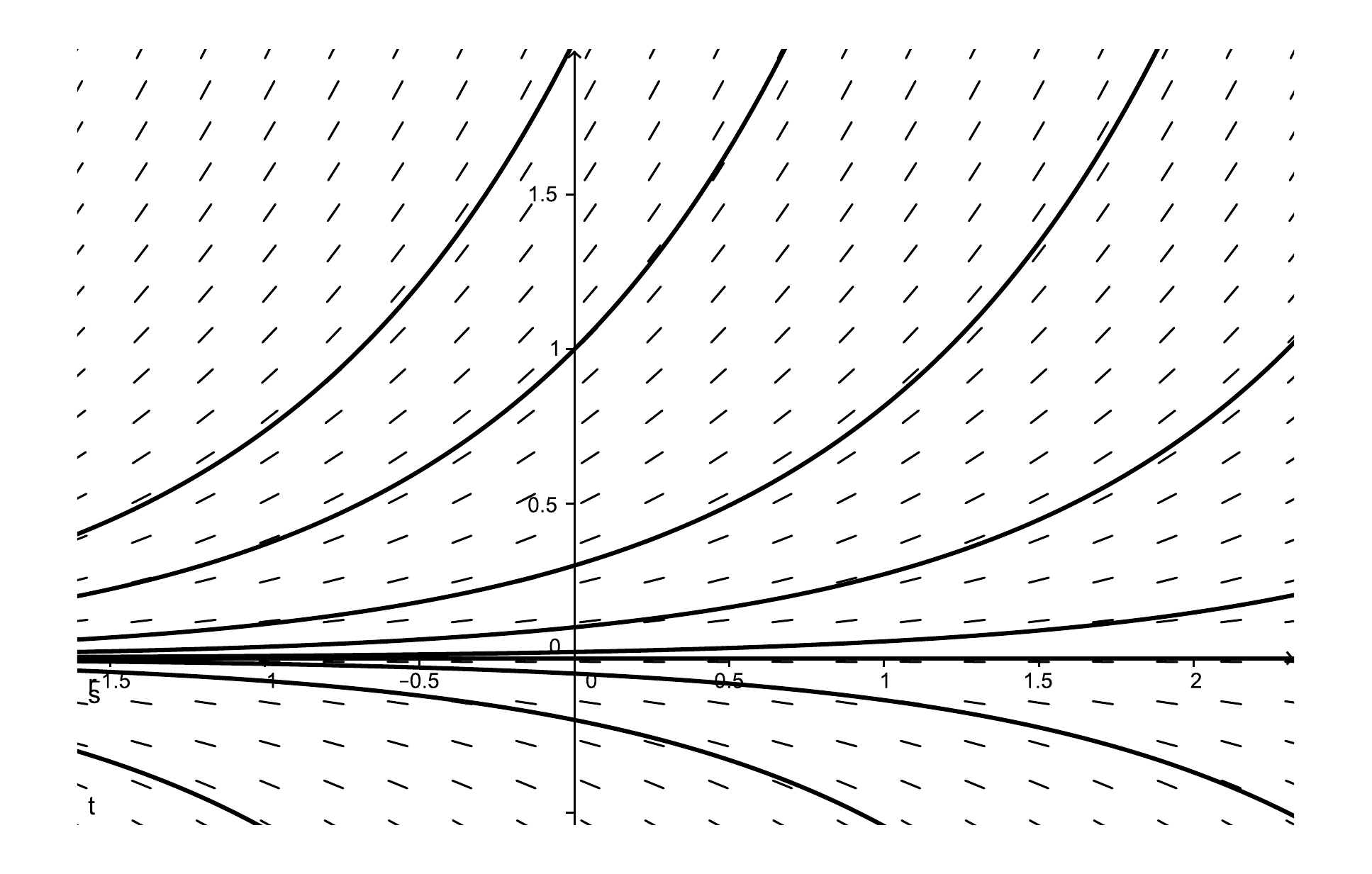}		
	\caption[A machine solving the differential equation $y'=y$]{A machine solving the differential equation $y'=y$ (left) and the relative slope field (right). 
	}\label{exp_machine}
\end{figure}

Until now, we passed from differential equation to TMM:
conversely now we convert the machine in differential polynomials. Given the machine of Fig. \ref{exp_machine}, what is the differential polynomial system defining its behavior according to the imposed constraints? 
First of all, let's define the machine more precisely.
\begin{ex}\label{example:exp}
Consider the machine defined by $P_2=(t,y)$ and the constraint that the direction of the point $(t,y)$ has to be the line passing through $(t-1,0)$. 
\begin{code}
    Consider the fixed points $P_0,P_1$.
     
    	\begin{tabular} {p{0.1cm} p{5.5cm} p{7.5cm}} 
		& \texttt{onRod(2,0,diff(2,1))} & consider $P_2=(t,y)$, thus $(t-1,0)$ is constrained to lie on the rod $r(2,0)$\\
		\noalign{\vskip 1mm}
		& \texttt{wheel(2,0)} & pose a wheel in $r(2,0)$ at $P_2$\\
		\noalign{\vskip 1mm}
	\end{tabular}
\end{code}
\end{ex}

Because of the wheel, $(t',y')$ has to be parallel to $r(2,0)$,
 therefore 
$t'y-y'=0$.
Note that the obtained differential equation is different from the original one ($y'-y=0$). The difference is given by the implicit assumption that $t'=1$. 
When we construct a machine solving a system of differential equations with the method seen in Prop. \ref{solving_diff_systems}, we implicitly assume that $t$ is the independent variable, so everything is obtained in function of its value. 
The introduction of a new variable (with constant derivative 1) for the independent one is a standard method to pass from a differential polynomial involving also the independent variable to an equivalent polynomial not depending directly on the independent variable. 

In summary, given a system of differential polynomials $\Sigma$, with the method of Prop. \ref{solving_diff_systems} we can construct a machine solving it, but the system $\Sigma^*$ obtained analyzing this machine is slightly different from the original $\Sigma$. If we want to obtain $\Sigma$ from $\Sigma^*$, we have to add the condition $x_i'=1$ for the variable $x_i$ that represents the abscissa of the independent point $(t,0)$. We can call such additional condition the ``{independentization of a variable}'' (because from $x_i'=1$ it follows $x_i=t+k$, i.e. $x_i$ is exactly the independent variable eventually translated of a constant $k$).

\subsection{Note on initial conditions}\label{initial_cond}
The total behavior of TMMs can be analytically defined by a system of differential polynomials. However, when a machine is considered to work on a plane, the initial position of its components can be considered implementing the initial conditions. 
In this section, we focus on how to apply these initial conditions.

Physical realizations of TMMs are devices that can be lifted and downed on the plane. While the device is not yet downed on the plane, there are fewer working constraints (because of the lack of wheel friction), so we can move some points that lose some degrees of freedom when wheels touch the plane.
Therefore, if we consider TMMs as physical devices, their assembly and use can be distinguished in two different steps:
\begin{enumerate}
	\item composition: the various parts are assembled in order to construct the machine;
	\item friction on the plane: the machine is ``put on the plane,'' so wheels avoid lateral motions.
\end{enumerate}
The difference between these two steps is the role of the wheel. In the first case the machine is constructed but, considering it lifted from the plane, the wheel constraints do not work, so on the machine only the holonomic constraints are active (the ones of algebraic machines). When we ideally put the constructed machine on the plane, wheels begin to have friction on the plane, and consequently the related nonholonomic constraints begin to work.

While the composed machine is already defining differential polynomial equations, the activation of the friction is related to the posing of initial conditions. 
In fact, in the instant when the constructed machine touches the plane (and the wheel friction begins), all the points have a certain position: the values of the variables relative to these positions can be viewed analytically as the initial conditions.
Therefore, to pose an initial condition to some variables, we have to suitably move the points (the position of which is related to the wanted variables) when the device is lifted. The downing of the device assures that the variables solve the Cauchy problem.

\begin{figure} 
	\begin{center}
		\includegraphics[scale=.25] {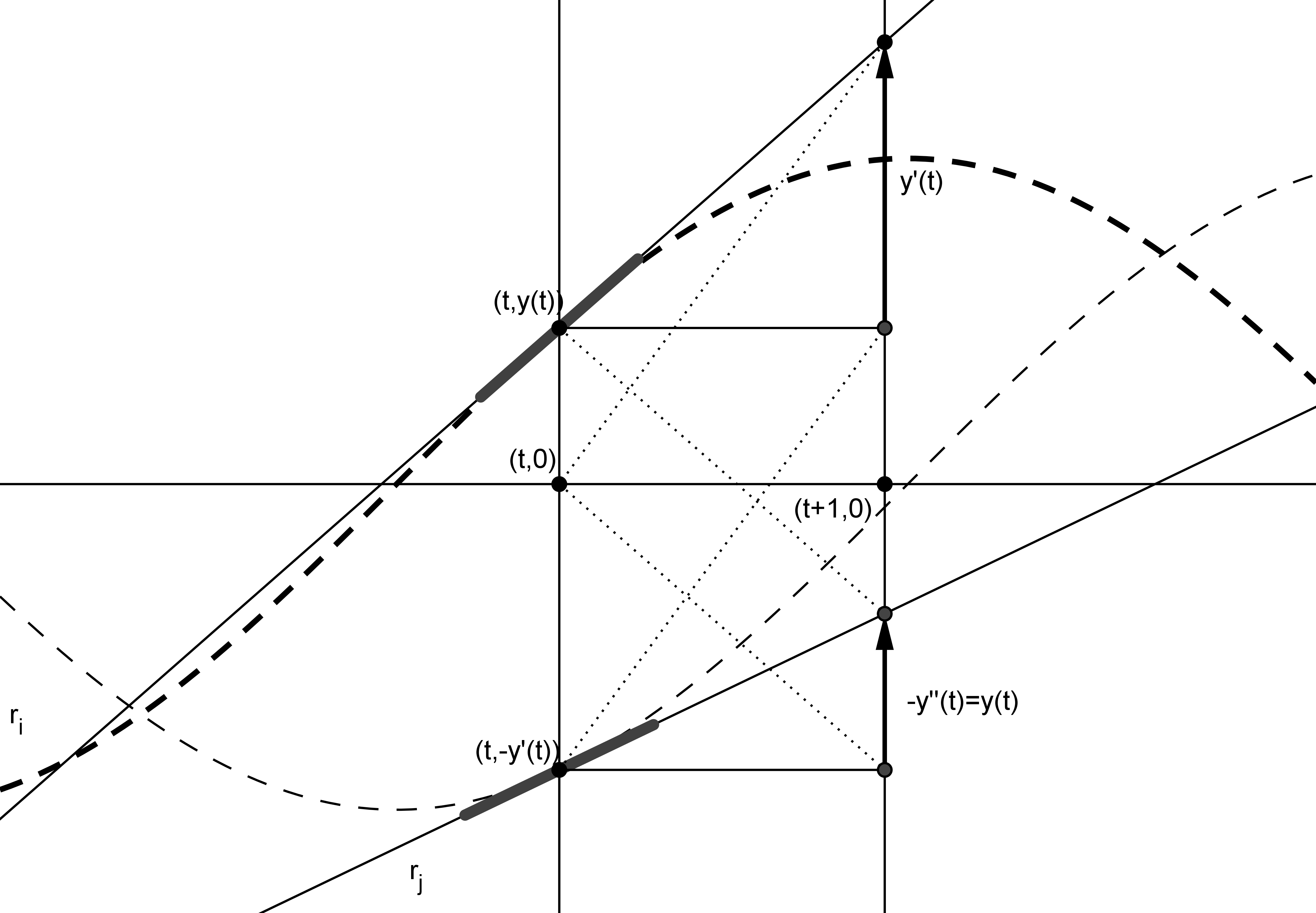}
	\end{center}
	\caption[A machine for $y=-y''$]{Sketch of a machine for $y=-y''$. Parallel dotted lines represent the translation of lengths.} \label{sin_machine}
\end{figure}

To clarify these ideas, as an example, we propose a machine solving the differential equation $-y''(t)=y(t)$.
According to different initial conditions, the same machine can generate the sine (posing $y(0)=0, y'(0)=1$) and the cosine function (with initial conditions $y(0)=1, y'(0)=0$).

As seen in Fig. \ref{diff_poly}, once introduced the point $(t,y(t))$, we can construct the point $(t+1,y(t)+y'(t))$. Reporting the length $-y'(t)$ as represented in the Fig. \ref{sin_machine} (parallel dotted lines represent the translation of lengths, without visualizing all the necessary components), we can construct $(t,-y'(t))$. Then, once constructed $(t+1,-y'(t)-y''(t))$, it is possible to impose $y(t)=-y''(t)$ by reporting the length $-y''(t)$.

Now it is time to impose initial conditions. For cosine requires totally similar steps, let's consider only the sine function, thus we have to impose $y(0)=0, y'(0)=1$. Physically, this condition has to be posed after the construction of the machine, and before the ``activation'' of the friction of the wheels. First, we move the cart in $(t,0)$ until it reaches the position $(0,0)$, then we move the carts $(t,y(t))$ and $(t,-y'(t))$ until they (respectively) reach the positions $(0,0)$ and $(0,-1)$. Once posed these conditions avoiding the nonholonomic constraints (ideally: when the machine is not yet put on the plane), the nonholonomic constraints of wheels can be activated (the machine can be finally put on the plane, allowing the friction of the wheels on the plane). In this way the machine generates exactly the sine function.

In contrast to the case of the exponential, the sine function is constructed using a second order differential equation, so it is not possible to consider the wheel solving a static graphical slope field. Indeed, the slope of the rod with a wheel is dynamically defined in function of the position of the other wheel. 

\section{Problem solving}\label{diff_problems}
With the introduction of tractional motion machines, we can overcome Cartesian geometry still relying on the idealization of suitable machines, and, thanks to differential algebra, we can also provide a well-defined language and set of algorithms for the analytical counterpart without the need of the infinity or of approximations (as the mathematical concept of limits). In this section, we suggest some applications of differential algebra for such machines.

According to \cite[p. 287]{Bos2001}, Descartes' geometric problem solving method consisted of an analytic part (using algebra to reduce any problem to an appropriate equation) and a synthetic part (finding the appropriate geometric construction of the problem on the basis of the equation).
Considering analysis by differential algebra and synthesis by TMMs, the same Cartesian problem-solving method can be extended beyond algebraic boundaries by the following steps: 
\begin{enumerate}
	\item start from a problem about TMMs,
	\item convert it in differential polynomials,
	\item solve the problem with differential algebra algorithms,
	\item when requested, after the simplification, find the specific solution with diagrammatic construction of TMMs.
\end{enumerate}
Regarding the third step, we suggest to manipulate equations with the \emph{DifferentialAlgebra} package of the computer algebra system Maple, of which we include commands in footnotes.

\subsection{The example of the cycloid}\label{analysis_cycloid}
As a first example, by automated reasoning we prove what has been informally observed in the section \ref{almost_cycloid} about the behavior of the machine of Ex. \ref{example:cycloid} (page \pageref{example:cycloid}). 

Consider $P_2=(t,0)$ and $P_{[1]}=(x,y)$ moving around $P_2$ at unitary distance, so 
\begin{equation}\label{cycloid_cond1}
(x-t)^2+y^2=1.
\end{equation}
The wheel in $P_{[1]}$ has as direction $P_{[1]}-P_{[0]}$, i.e. $P'_{[1]}=(x',y')$ has to be parallel to $P_{[1]}-P_{[0]}=(x-t,y-1)$:
\begin{equation}\label{cycloid_cond2}
y'(x-t)=x'(y-1).
\end{equation}

Thus, we have two equations (the first purely algebraic and the second differential) in $t,x,y$. If we are interested in the curve traced by $P_{[1]}$, we can use differential elimination to eliminate $t$. We can proceed with the following steps:
\begin{enumerate}
	\item consider the differential ring $R$ with rational coefficients having as variables $t,x,y$, and adopt a ranking eliminating $t$;
	\item consider the ideal $I$ in $R$ generated by the two differential polynomials;
	\item consider in $I$ the \emph{differential regular chains} eliminating $t$.
\end{enumerate}
We can translate these steps in commands for computer algebra software\footnote
{In Maple we can perform these operations with the following code lines (commented on the right):\\
	
	\begin{tabular} {p{7cm} p{6.5cm}} 
		\texttt{with(DifferentialAlgebra)}	& load the package\\
		\noalign{\vskip 1mm}
		\texttt{R := DifferentialRing(blocks=[t,x,y], derivations=[a])} & construct the differential ring with as independent variable $a$, and dependent ones $t,x,y$ with the ranking $t\gg x \gg y$\\
		\noalign{\vskip 1mm}
		\texttt{p := (x(a)-t(a))\^{}2+y(a)\^{}2 = 1} & $p$ is an algebraic equation\\
		\noalign{\vskip 1mm}
		\texttt{q := (diff(y(a), a))*(x(a)-t(a)) = (diff(x(a), a))*(y(a)-1)} & $q$ is a differential equation (\texttt{diff(f(a), a)} stands for the derivative $df/da$) \\
		\noalign{\vskip 1mm}
		\texttt{ideal := RosenfeldGroebner([p, q], R)} & \texttt{ideal} is the radical differential ideal generated by $p$, $q$ \\
		\noalign{\vskip 1mm}
		\texttt{Equations(ideal)} & returns the equations of \texttt{ideal}\\
		\noalign{\vskip 1mm}
		\texttt{Inequations(ideal)} & returns the inequations of \texttt{ideal}\\
		\noalign{\vskip 1mm}
	\end{tabular}
	
	Note that the commands \texttt{Equations(ideal)} and \texttt{Inequations(ideal)} 
	show the differential regular chains for the ideal in $t,x,y$.
	
	Once obtained the differential regular chains reduced with respect to a certain ranking, the elimination of the greater depending variable only consists in taking all and only the equations and inequalities of the differential regular chains where the variable and its derivatives do not occur. Using Maple, it can be achieved with the command: \texttt{Equations(ideal, leader < t(a))}.
}. In particular we obtain that the \emph{differential regular chains} (for the ideal generated by the two equations characterizing the TMMs) reduced with the ranking $t\gg x \gg y$ are:
$$ C_1 =\{  t y' +x' y -x'-x y'= 0,  x'^2 y - x'^2 + y'^2 y + y'^2 = 0,
y'\neq 0,     x' y - x' \neq 0,   y-1 \neq 0  \};$$
$$ C_2 = \{ t-x = 0, y-1 = 0 \};$$
$$ C_3 = \{ t^2 - 2x t + x^2 + y^2 -1 =0;   x'=0,   y'=0,   t-x\neq 0\};$$
$$C_4 = \{  t-x=0, x'=0, y^2-1=0,   y \neq 0    \}.$$
But, as said, we are not interested in the behavior of $t$, so, if we eliminate it (considering the given ranking of the variables, we can just skip all the equations with $t$ in $C_1, C_2, C_3, C_4$), we obtain
$$C_1^*=\{  x'^2 y - x'^2 + y'^2 y + y'^2 = 0,  y'\neq 0,   x' y - x' \neq 0,   y-1 \neq 0  \};$$
$$C_2^*=\{  y-1 = 0 \};$$
$$C_3^*=\{x'=0, y'=0	\};$$
$$C_4^*=\{x'=0, y^2-1=0, y\neq 0
    \}.$$
Even though it is possible to do some simplifications, we adopted the given form (that is exactly the one given by the Maple code) to evince the fact that any reasoning can be conducted in a purely formal way without considering the semantic meaning.
We can observe that $C_3^*$ and $C_4^*$ does not provide us anything interesting but single points.

On the contrary, we can observe that $C_1^*$ contains as equation the general solution which, rewritten as an ODE, becomes the differential equation of the cycloid: 
$$\left( \frac{dy}{dx} \right)^2=\frac{1-y}{1+y}.$$
Another solution out of arcs of cycloids is made up by the line $y= 1$, that was excluded in $C_1^*$ because of its inequalities.

We can also ask ourself which constraints can be added to construct exactly a cycloid. 
According to the property seen in Fig. \ref{cycloid-dim} (page \pageref{cycloid-dim}), we can impose a new tangent condition in another point of the circumference of the rolling disk. As such a point, we can consider the point symmetric to $P_{[1]}$ with respect to the center $P_2$, as visible in Fig. \ref{cycloid-pure}. 
\begin{figure}
	\center \includegraphics[scale=.2] {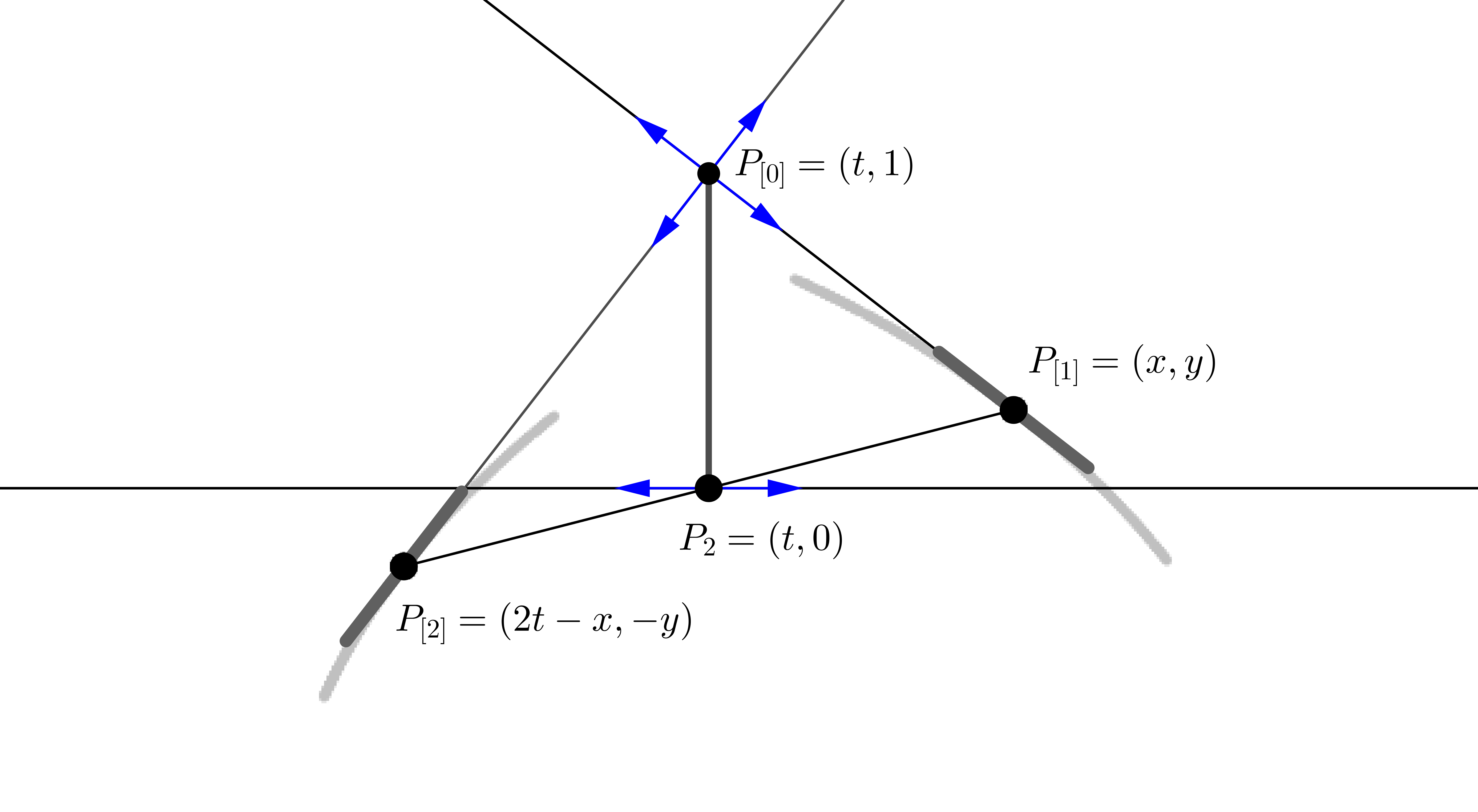}
	\caption[A machine for the cycloid]{A machine for the cycloid. We introduce a new point $P_{[2]}$ symmetric of $P_{[1]}$ with respect to $P_2$. Also in $P_{[2]}$ we pose a wheel on the rod passing through $P_{[0]}$.}\label{cycloid-pure}
\end{figure}
Such machine can be defined by appending few instructions to the code of the example \ref{example:cycloid}: 
\begin{code}
    We add a new point $P_{[2]}$ and pose a wheel on it. 

    	\begin{tabular} {p{0.1cm} p{7.5cm} p{6.5cm}} 
		& \texttt{dist(2,0,-1,[2])} & the point $P_{[2]}$ is introduced as the symmetric of $P_{[1]}$ with respect to $P_2$\\
		\noalign{\vskip 1mm}
		& \texttt{onRod([2],0,[0])} & the point $P_{[0]}$ is constrained to lie on the rod $r([2],0)$\\
		\noalign{\vskip 1mm}
		& \texttt{wheel([2],0)} & pose a wheel in $r([2],0)$ at $P_{[2]}$\\
		\noalign{\vskip 1mm}
	\end{tabular}
\end{code}
The wheel in $P_{[2]}=(2t-x,-y)$ has as direction $P_{[2]}-P_{[0]}$. That means that $P'_{[2]}=(2t'-x',-y')$ has to be parallel to $P_{[2]}-P_{[0]}=(t-x,-y-1)$:
\begin{equation}\label{cycloid_cond3}
(2t'-x')(-y-1)=-y'(t-x).
\end{equation}
If we consider the ideal generated by the three equations (\ref{cycloid_cond1}), (\ref{cycloid_cond2}) and (\ref{cycloid_cond3}), we can compute the relative differential regular chains eliminating $t$. We obtain that $x$ and $y$ have to satisfy the differential systems $C_1^{**}, C_2^{**}, C_3^{**}$ where $C_1^{**}=C_1^*, C_2^{**}=C_3^*, C_3^{**}=C_4^*$, i.e. there is no longer the solution $y=1$ given by $C_2^*$. Thus, with the new conditions, the point $(x,y)$ is always constrained to walk along a cycloid.

\subsection{Differential universality theorem and constructible functions} \label{constructible_functions}
It is time to give a characterization theorem for the behavior of TMMs, somehow the extension of Kempe's (algebraic) universality theorem to the differential case.

\begin{thm}[Differential universality]\label{diff_universality}
    Considering $n$ variables, the behavior of a TMM coincides with 
    the union of the solutions of a finite number of systems of differential polynomial equations and inequations with integer coefficients and with real functions as independent variables.
\end{thm}
\begin{proof}
    We split the theorem in two: 1. the machine behavior is the union of the systems; 2. for any finite set of systems, we can consider a machine with this behavior.
    
    1. By Prop. \ref{TMM-total_behavior}, the total behavior of any TMM is the solution of a differential polynomial system $\Sigma$ in $m$ variables ($m\geq n$). Thus, to restrict to $n$ variables, we can use the Rosenfeld-Gr\"obner algorithm of section \ref{differential_algebra}, hence the behavior can be given as a finite union of systems of differential-polynomial equations and inequations.
    
    2. With certain modifications, we can reuse some ideas in the proof of theorem \ref{alg_universality} (algebraic universality): for every inequation $q_i(x_1,\ldots,x_n)\neq 0$ we can introduce a new variable $t_i$ and the differential polynomial $\tilde{q_i}(x_1,\ldots,x_n,t_i)=q_i\cdot t_i-1$. Thus, even though with more variables, we got that every system can be considered as made up by only differential polynomial equations with integer coefficients.
    Then, we can use the peculiarity of the variables of TMMs to be real functions: any system of real differential polynomials $p_1=\ldots=p_l=0$ is equivalent to a single polynomial $(p_1)^2+\ldots+(p_l)^2=0$. Let $P_1, \ldots, P_k$ be the differential polynomials identifying the $k$ systems:  the union of their zero sets has to satisfy the polynomial $P= P_1\cdot \ldots \cdot P_k$. By Prop. \ref{solving_diff_systems}, the zero-set of $P$ can be constructed by a TMM. Thus, restricting on the variables $x_1,\ldots, x_n$, we got a machine with the sought behavior. 
\end{proof}
While theorems of algebraic universality and differential universality have many similarities, we have to highlight that in the algebraic case we have inequalities while in the differential only inequations. Remind that differential algebra does not distinguish between real or complex or other kind of indeterminates (Rosenfeld-Gr\"obner algorithms works for any differential ring), while semi-algebraic sets are defined specifically for the real case. 



We can also define the nature of the functions that TMMs can generate. After the definition of \textit{differentially algebraic functions} \cite[p. 777]{Rub1989}, we use it to give a classification of the constructible functions. This result is particularly interesting from an historical perspective: differently from the algebraic case, the classification of the curves traced by tractional motion was previously missing.
\begin{defn} \label{DA_def}
    A function $y$ is \emph{differentially algebraic} (shortly: D.A.) if it satisfies an algebraic differential equation, i.e. a differential equation in the form $P(t,y,y',\ldots,y^{(n)})$ where $P$ is a nontrivial polynomial in $n+2$ variables. 
\end{defn}
The non-triviality condition is essential because every function is solution of $0=0$. 

\begin{prop}[constructible functions]\label{constructible_funs}
    The curves generated by a TMM are all and only the image of D.A. functions.
\end{prop}
\begin{proof}
    Given a certain TMM, consider a point $(x,y)$ of the machine with one degree of freedom (if the point has two or more degrees of freedom it does not define a curve). 
    The curves traced by this point are defined as all the value of $x,y$ satisfying the restricted behavior, i.e. many systems of differential polynomial equations and inequations in $x,y$. Now we can be interested in interpreting a curve as the graph of a function (at least locally). So, we can consider the curve as a function $y=f(x)$ [respectively $x=g(y)$ in a neighborhood of vertical tangents; however, we will no longer consider this case because we only have to switch the role of $x$ and $y$]. 
    To achieve this aim, we can no longer consider $x$ as a dependent variable, but as an independent one. Algebraically, this is translated (as seen in the section \ref{independentization}) by the ``{independentization condition}'' $x'=1$. Indeed, if we add the new condition to the systems in $x,y$, we can again consider the elimination of $x$ obtaining a family of differential regular chains only in $y$. Thus, we find that the curves $(t,y(t))$ are (locally) solutions of differential polynomials in $y$, so D.A.

    Conversely, we have to recall that any D.A. function satisfies an algebraic differential equation \textit{with integer coefficients} \cite[Th. 1, p. 778]{Rub1989}. 
    As observed in section \ref{independentization}, the introduction of a new variable with constant derivative 1 for the independent one allows the construction of an equivalent polynomial not depending directly on the independent variable. Once eliminated the new variable, the obtained differential polynomial in $y$ can be solved by a TMM. Hence, the constructible functions are all and only the {differentially algebraic} ones. 
\end{proof}

This result is important because it means that TMMs generate a new dualism beyond algebraic/transcendental (and this time about functions, not curves or algebraic varieties as done with algebraic machines). Note, however, that a machine can construct functions that globally are not D.A., as visible since the first introduction of TMMs (\cite{Mil2012} concerned the construction of a machine tracing a cycloid, that globally is not D.A.), but locally each of these functions has to be D.A.

All the elementary functions are D.A., and even most of the transcendental functions that we find in analysis handbooks. Historically, the first example of non-D.A. function was the $\Gamma$ of Euler, as proven in \cite{Hol1886}. Note that $\Gamma$ function is not even locally D.A., that is why it cannot be constructed by TMMs.
\\

As an example, we can continue with the cycloid, observing some differences when we ``{independentize}'' different variables.

Adding the constraint $x'=1$ to the equations (\ref{cycloid_cond1}), (\ref{cycloid_cond2}) and (\ref{cycloid_cond3}), we can consider $y$ in function of $x$. This time, with a ranking eliminating $t$ and $x$, we obtain only one regular chain:
$$C_{\{x'=1\}}=\{
y'^2 y + y'^2+y-1=0;
y' y + y'\neq 0;
y+1\neq 0\}.$$

This representation is not useful to identify the traced curve as the usual parametrization of a cycloid. This identification is more visible if we ``{independentize}'' another variable. Consider the additional constraint $t'=1$ instead of $x'=1$. Even in this case, we obtain only one regular chain that, upon eliminating $t$, becomes:
$$C_{\{t'=1\}}=\{x'-y-1=0; y'^2+y^2-1=0; y'\neq 0\} $$

Now we can observe that this representation is the one of 
$$
\begin{cases}
x=t+\cos t\\
y=-\sin t
\end{cases}
$$
Indeed, instead of the trigonometric functions we can convert the system in a purely differential polynomial one:
$$
\begin{cases}
x'=1+y\\
y'^2+ y^2=1\\
y''=-y
\end{cases}
$$
Computationally, we can check that it has as regular chains exactly $C_{\{t'=1\}}$ (in both cases the computed regular chain is $\{ y - x' + 1 = 0;   x''^2 + x'^2 - 2x' = 0; x''\neq 0 \}$). 

Obviously different machines can construct the same manifold of zeros. Remaining on the example of the cycloid, we can construct a TMM having a point of coordinate $(x,y)$ satisfying the equations of $C_{\{t'=1\}}$ by the standard method seen in Prop. \ref{solving_diff_systems} and Fig. \ref{diff_poly}. This way, we consider separately the variables $x$ and $y$ by the introduction of the points $(t,x)$ and $(t,y)$, impose algebraic and differential conditions on both dependent variables, and then we construct the point having as coordinate $(x,y)$. This method is general, but of course, does not provide the simplest machine (we do not intend to study the notion of simplicity, consider ``simplest machine'' in an intuitive sense).
%

\subsection{Equivalence between TMMs}\label{equiv_diff_machines}
In the previous example, we have seen that two radical ideals were equivalent because they had the same representation. However, the opposite in general does not hold.

As seen in the section \ref{full_behavior_diff_machines}, the total behavior of TMMs is the solution of a system of differential polynomial equations, so the restricted behavior is the restriction to the relative manifold of solutions on some variables. Before checking the equality test between two machines on certain variables, we have to suppose that the variables of the restricted behavior are in the same number in the two manifolds.

Let $X=\{x_1,\ldots,x_n\}, Y= \{y_1,\ldots,y_m\}, Z=\{z_1,\ldots,z_l\}$ be sets of dependent indeterminates. Consider the radical differential ideal $A$ generated by the differential polynomials $p_1, \ldots,p_i$ on $X\cup Y$, and consider $B$ generated by $q_1,\ldots,q_j$ on $X\cup Z$.

In general, there is no known method to check the equality of two radical differential ideals
represented by regular differential chains
(it is related with the so called \textit{Ritt's open problem} \cite{GKO2009}), but in this case we have much stronger hypothesis. In fact, we known the generators of the total behavior: with this condition the solution is easily achievable and computable.

To check the equality between two total behaviors (i.e. between radical differential ideals given by a finite set of generators), we can fix a certain ranking and compute the regular differential chains using the Rosenfeld-Gr\"obner algorithm, and then we can test whether all the generators of the first ideal belong to the second and vice-versa\footnote
    {According to the notation of this section, in the case of total behaviors $Y$ and $Z$ are empty sets of variables. We can test whether the ideals $A$ and $B$ are equal using the Maple command \texttt{BelongsTo} of the \emph{DifferentialAlgebra} package. Once given any ranking, and constructed with \texttt{RosenfeldGroebner} the ideals $A$ and $B$ by their generators, to check the equality we only have to test whether all the generators of $A$ belongs to $B$ and vice-versa. In Maple, the command \texttt{BelongsTo([$p_1,\ldots,p_n$],B)} produces as output a list of $n$ true/false, the $i$-th of which indicates whether $p_i$ belongs or not to $B$. Conversely \texttt{BelongsTo([$q_1,\ldots,q_n$],A)} can be used to check the belonging to $A$.
}.

But we are interested in behaviors obtained by eliminating some variables, that are in general represented by an intersection of families of regular chains, and there is no known algorithm to pass from a representation of families of regular chains to a list of generators. 

Given the ranking $Z\gg Y\gg X$ (or $Y\gg Z\gg X$), compute the families of regular chains $R_A$ and $R_B$ representing respectively $A$ and $B$ (e.g using Rosenfeld-Gr\"obner's algorithm).
Let $R_A^*$ be $R_A$ without the equations/inequations involving $Y$ and, similarly, let $R_B^*$  be $R_B$ eliminating $Z$.
We can verify the inclusion of $A$ in $B$ restricted to $X$ by checking whether all the $p_i$ belongs to $R_B^*$; similarly for the opposite inclusion check whether $q_i$ belongs to $R_A^*$. If both the inclusions are verified, the two systems are equivalent restricted to the variables $X$.
\\

Note that we have treated TMMs without any reference to initial conditions. As far as my knowledge goes, the equality problem is still open if we introduce initial values (cf. note \ref{note:initial_values} at page \pageref{note:initial_values}). With regard to some positive results, we can consider \cite{BR2012}, which provides an algorithm for the symbolic solution of linear boundary problems, passing from differential algebra to \textit{integro-differential algebras} (Green's operators). 

Even though without any final answer about initial conditions problems, we want to underline that differential algebra permits to check the equivalence between TMMs. On the other side, even in the more concrete approach to calculus, \textit{computable analysis} \cite{Wei2000}, it is not possible to check the equality test between any general couple of generated objects (i.e. computable numbers). 
Considering intuitively ``exactness'' as the property of a computational frameworks (both in analytic or geometric paradigm) to be independent from non-finitary procedures (as unlimited approximations), we think that it will be interesting in future to deepen the relation between the computability of the equality test in a theory and the exactness of the theory.

\subsection{Differential machines equivalent to algebraic ones}\label{diff_as_alg_machines}
Consider a TMM defining the motion of a point $P$ of coordinates $(x,y)$ so that the tangent in $P$ is perpendicular to the line passing through $P$ and the point $(x+1/2,0)$ as the Fig. \ref{sqrt_machine}  shows. This machine is defined by the differential polynomial $x'-2y y'$, which is the total derivative of $x - y^2$. Therefore, for every constant $c\in\mathbb R$, solutions are parabolas satisfying $x=y^2+c$. That means that we are able to trace any of the solutions of this TMM with an algebraic one. Hence, the general question arises: can we characterize the TMMs having solution constructible with algebraic machines (by eventually adding a finite number of real constants of integration)?
\begin{figure}
	\center 
	\includegraphics[scale=.5] {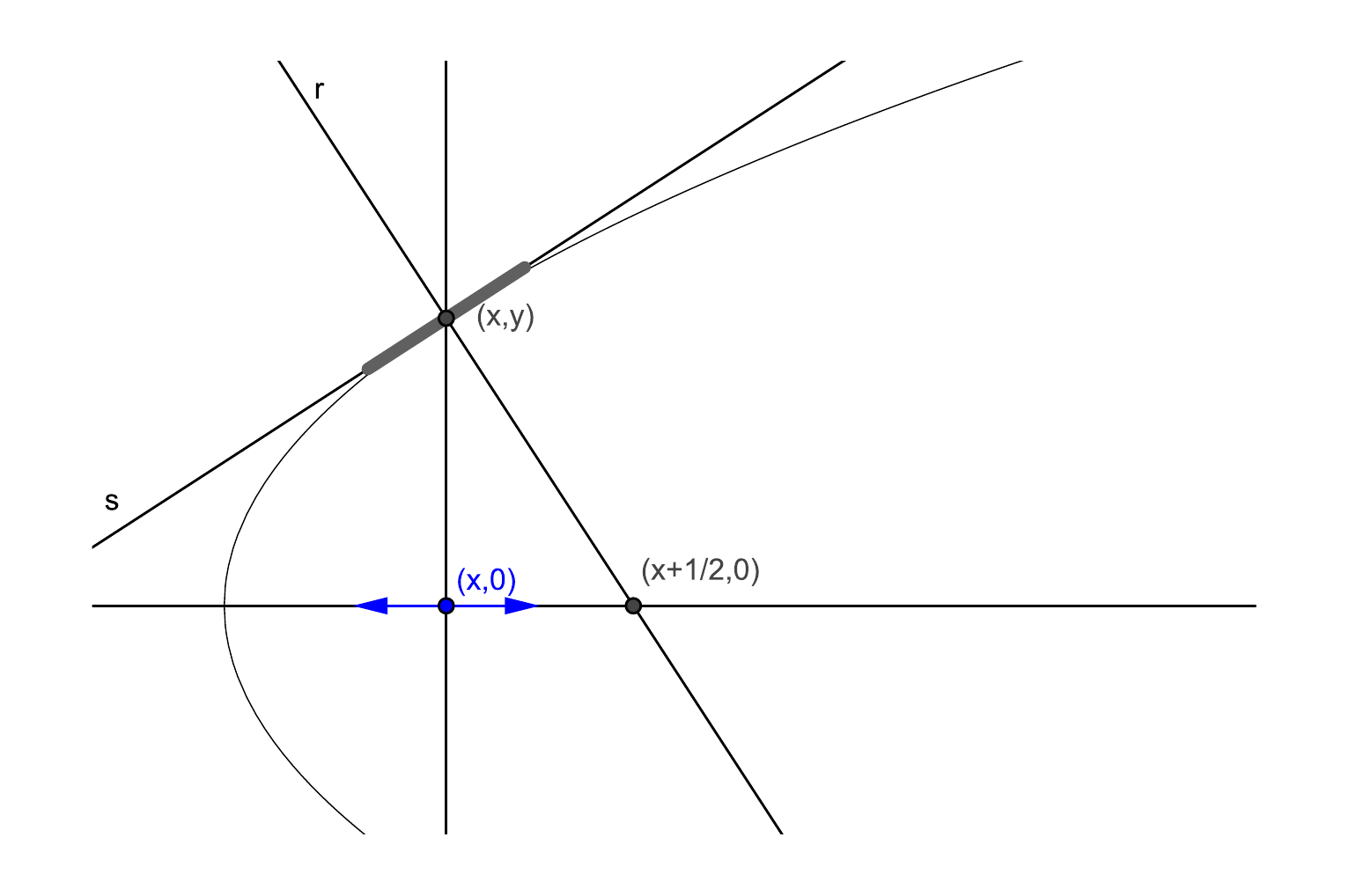}
	\caption[A machine with the tangent in $(x,y)$ perpendicular to the line passing through $(x+1/2,0)$]{Sketch of a machine with the tangent in $(x,y)$ perpendicular to the line passing through $(x+1/2,0)$. To simplify the visualization, only the cart for $(x,0)$ is shown.
	}\label{sqrt_machine}
\end{figure}

Given a differential system or even its restriction on some variables, to find the algebraic constraints satisfied we can simply use the orderly ranking in the Rosenfeld-Gr\"obner algorithm. There are algebraic constraints if and only if in the obtained family of regular differential chains there are polynomial equations without any proper derivative (i.e. of order 0)\footnote
{In Maple, given the dependent variables $x_1,\ldots,x_n$ and the independent variable $t$, one can construct a differential ring with the orderly ranking by the command
	\texttt{R := DifferentialRing(blocks = [[$x_1,\ldots,x_n$]], derivations = [t])} (the double square brackets \texttt{[[$\ldots$]]} indicate the orderly ranking). After the usual construction of the ideal \texttt{ideal} with the Rosenfeld-Gr\"obner algorithm, the purely algebraic constraints are given by \texttt{Equations(ideal, order=0)}.
}.

It is more complicated if we are interested not only in algebraic constraints, but also on first integrals given by algebraic constraints. Given a system $\Sigma$ of ODEs, a \textit{first integral} is a function $f(t)$ whose value is constant over the independent variable $t$ along every solution of $\Sigma$. 
To find such first integral, we can use the algorithm \textit{findAllFirstIntegrals} proposed in \cite{BL2015}. The algorithm takes as input a family of systems of differential polynomial equations and inequations, and a set of monomials $\mu_i$ in the variables and their derivatives ($i=1,\ldots,n$); it returns as output the coefficients $\alpha_i$ s.t. the differential polynomials $\sum_{i=1}^n \alpha_i \mu_i$ are first integrals of $\Sigma$.

If in the algorithm we consider as input the behavior of a TMM (a family of regular chains) and as $\mu_i$ the various combination of the variables (not their derivatives), we can obtain all the algebraic first integrals of any fixed degree (e.g. given the variables $x,y$, to check all the algebraic first integrals up to second degree we can consider as $\mu_i$ the monomials $x, y, x^2, xy, y^2$). However, while for any degree we can suitably consider the monomials, at my knowledge there is no general method to check whether the behavior allows any algebraic first integral (of arbitrarily high degree).

\subsection{3D TMMs}
We considered TMMs working on a plane, but what about the extension of these machines beyond the planar behavior? Here we provide a sketch for the possible physical implementation of a 3D TMM and some shallow explanations about its behavior. Physically, we can introduce a cube of gelatinous material: on it we can hypothesize that thin rods can freely move, while a small disk of center $C$ has to locally represent the tangent plane to the surface walked by $C$. Similarly to 2D TMMs, the main idea is to set some constraints to the tangent: in this case consider $u(x,y)$ (from now on we implicitly assume the dependence on $x,y$)  as the function to be found and, as usual in partial differential equations, let $u_x=\frac{\partial u}{\partial x}$, $u_y=\frac{\partial u}{\partial y}$.
Considering the point $C=(x,y,u)$, the disk centered in $C$ has two perpendicular rods on the tangent plane passing respectively through $(x+1,y,u+u_x)$  and $(x,y+1,u+u_y)$, as visible in Fig. \ref{fig:3d}. That imposes the suitable partial derivative conditions $u_x, u_y$, and the lengths $u_x, u_y$ can be used to set other conditions (as in 2D TMMs).

The study of such machines could be interesting as a future perspective.
However, even if it is far away from the main aims of this work, we give an informal justification about the idea that 3D TMMs do not generate any function $\mathbb R \rightarrow \mathbb R$ out of the ones constructible by 2D TMMs (it is not a complete proof, we do not deal with coefficients or with a sufficiently deepened characterization of 3D TMMs).

Looking for a characterization of the behavior of such machines, considering as variables smooth $\mathbb R^2 \rightarrow \mathbb R$ functions, we can perform their sum, multiplication and derivation. Differently from the 2D case, this time the derivation is with respect to two different independent variables ($x,y$).
Even with these machines, the suitable analytical tool is differential algebra, but with partial derivatives: such development is present since Ritt's works \cite[Ch. IX]{Rit1950}. 
The elimination is still available, so a constructible function $u: \mathbb R^2 \rightarrow\mathbb R$ defined by a 3D TMM has to be solution of non-trivial differential polynomials in $u$ and a finite number of its partial derivatives. 

To compare constructible functions between 3D and 2D TMMs, we have to restrict somehow the functions of the 3D case. Specifically, the graph of a unary function can be obtained intersecting the graph $G$ of $(x,y,u(x,y))$ with a plane $\alpha$ perpendicular to $OXY$. 
The intersection between $\alpha$ and $OXY$ is a line: let its natural parametrization be $x=\frac{a}{\sqrt{a^2+b^2}}t+a_0, y=\frac{b}{\sqrt{a^2+b^2}}t+b_0$ (as natural parametrization we mean that the derivative of the parametrization has unary length).
Thus, called $F$ the real function whose graph is defined by the intersection of $\alpha$ and $G$, $F(t)=u\left(\frac{a}{\sqrt{a^2+b^2}}t+a_0, \frac{b}{\sqrt{a^2+b^2}}t+b_0\right)$. We want to show that $F$ is a D.A. function.

Consider the function $f(x,y)=u\left(\frac{a}{\sqrt{a^2+b^2}}x+a_0, \frac{b}{\sqrt{a^2+b^2}}x+b_0\right)$ ($y$ is not used in its computation). We cannot convert such $f$ directly with differential polynomials in partial differential algebra (differential algebra doesn't deal with arguments of functions), but we can compute the partial derivatives of $f$. Trivially $f_y =0$, and, using the chain rule for the derivation of compositions involving multivariable functions, $f_x = u_x \frac{d}{dx}\left(\frac{a}{\sqrt{a^2+b^2}}x+a_0\right) + u_y \frac{d}{dx}\left(\frac{b}{\sqrt{a^2+b^2}}x+b_0\right) = \frac{a u_x + b u_y}{\sqrt{a^2+b^2}}$.
Thus, we can finally add the differential polynomial conditions for $f_x, f_y$ to the polynomials defining the total behavior of the 3D TMM. By the elimination of all the dependent variables out of $f$, we get polynomials on $f$ and its partial derivatives. But, being $f_y=0$, we can simply rename 
$f(x,y)=F(x), f_x(x,y)=F'(x), f_{xx}(x,y)=F''(x), \ldots$, thus obtaining a differential polynomial just on $F$ and its ordinary derivatives, i.e. $F$ has to be a D.A. function. 
\begin{figure}
    \centering
    \includegraphics[width=.4\textwidth]{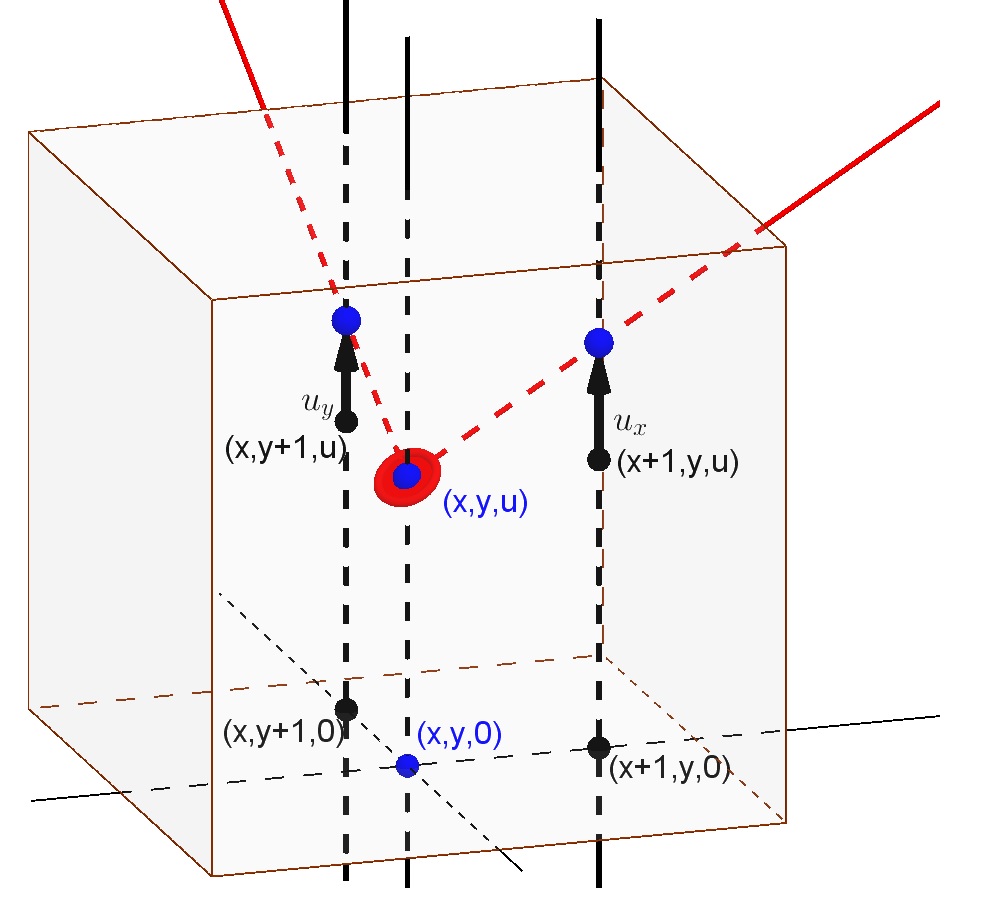}
    \caption{Sketch of 3D TMMs. To set partial derivative conditions on $u(x,y)$ we  can consider a jelly cube. Inside, a small disk centered in $(x,y,u)$ is constrained by two rods to lie on the plane passing through $(x+1,y,u+u_x)$  and $(x,y+1,u+u_y)$. That imposes $u$ to satisfy certain values of $u_x$ and $u_y$.}
    \label{fig:3d}
\end{figure}


\section{Conclusions and future perspectives}\label{balance}
%
\subsection{Extending Descartes' balance}
The richness of Cartesian setting depends on the correspondence between objects of the analytical and the synthetic part. 
From this perspective, the role of suitable ideal machines was central. 
The balance between machines, algebra, and geometry as suggested by Descartes was historically broken by the increase in importance of the analytical part with respect to geometric constructions. In particular, infinitesimal analysis introduced infinitary tools in the analytical part such as series or infinitesimal elements. However, even though with some centuries of delay, we can consider the finite approach to calculus objects of differential algebra as a legitimate descendant of polynomial algebra. Contrarily, the synthetic part can be managed with the proposed TMMs, which, as a well-defined model for tractional constructions, can be considered as an extension of Descartes' machines. The surprising result is that these heirs of Descartes' analytical and synthetic tools are still in balance (\textit{differential universality theorem}), being the behavior of TMMs exactly the space of solutions definable with differential algebra (restricted to ordinary differential equations). Furthermore, to study the properties of the machines there is no need of geometric or analytic insights, because the algebraic part provides algorithms to compute them autonomously (automated reasoning).

In this paper, we have been able to define the behaviors of TMMs that have been introduced to formalize tractional constructions in a modern way. To my knowledge, it is the first clear definition of the limits of tractional motion. Such limits permit a distinction between objects that are constructible with TMMs and others that are not. To define the behavior of such machines, we used manifolds of functions: if Descartes' setting defined a dualism between algebraic and transcendental curves, our setting facilitates a new dualism between functions. As introduced in section \ref{constructible_functions}, the obtainable functions are the \textit{differential algebraic} ones (shortly: D.A.), i.e. solutions of algebraic differential equations\footnote
{
	Note that, even though without the geometrical interpretation, such dualism was introduced in analog computation by Shannon's \textit{General Purpose Analog Computer} \cite{Sha1941} many centuries after the introduction of tractional constructions. As visible in the title of Shannon's paper, the GPAC was a theoretical model for the analog computer called \textit{differential analyzer}.
	
	Furthermore, we also need a note on an implicit assumption. 
	We considered function locally smooth (i.e. of class $C^\infty$ in a certain domain). In general, to be a solution of an algebraic differential equation of order $n$, we have to assume that the function is $C^n$, not necessary $C^\infty$. For not-smooth functions, most of the results of differential algebra 
	fail.
	These cases are treated in \cite{Rub1983}. 
}.
As already mentioned, all elementary functions are D.A., and even most of the transcendental functions that we find in most of the analysis handbooks. Historically, the first example of non-D.A. function was the $\Gamma$ of Euler, as proven in 1886 by H{\"o}lder. 
According to \cite{Rub1989}, not-D.A. functions are named \textit{transcendentally transcendental}. 
About functions in one variable, we can provide a finer distinction beyond the algebraic and transcendental dualism. Of course algebraic functions are also D.A., so, calling \textit{algebraic-transcendental} the functions that are D.A. but not algebraic, we can divide functions in the cases of Table \ref{fun_cases} (with some examples).
\begin{table}
	\caption[Categorization of functions in one variable.]
	{Categorization of functions in one variable (taken from \cite[p. 501]{Sha1941}).}
	\label{fun_cases}
	\begin{tabular}{|c|c|}
		\hline
		\textbf{Transcendental} & \textbf{Algebraic}\\
		\hline		
		\begin{tabular}{p{2.5cm}|p{2.7cm}}
			\textit{Trascendentally trascendental}
			&
			\textit{Algebraic-trascendental}
			\\
			\hline
			Euler $\Gamma$, Riemann $\zeta$
			&
			$e^x, \log(x)$; trigonometric, hyperbolic and inverses; Bessel, elliptic and probability functions			
		\end{tabular}
		&
		\begin{tabular}{p{2.5cm}|p{1.7cm}}
			\textit{Irrational algebraic}
			&
			\textit{Rational}
			\\
			\hline
			$x^m$ ($m$ a rational fraction); solutions of algebraic equations in terms of a parameter
			&
			polynomials, quotients of polynomials
		\end{tabular}
		\\
		\hline
	\end{tabular}
\end{table}

\subsection{Beyond TMMs} \label{Beyond_TMMs}
About future extension beyond TMMs, a crucial role could be played by Euler $\Gamma$ function. As we are going to examine, this function can play for D.A. functions the same role as the exponential curve played for algebraic curves.

Algebraic curves are defined as the zero set of polynomials, where a polynomial is an expression 
that involves only the operations of addition, subtraction, multiplication, and non-negative integer exponents. One can ask to relax the constraint of considering only non-negative integer exponents (e.g. we may be interested in considering monomials in $x^{-\frac 3 2 }$, or in $x^{\sqrt 5}$):
from this perspective, the exponential curve solves the problem of generic exponent.
Concerning constructions, tractional motion justified the exponential curve with the introduction of loads subject to friction or with blades or wheels. 
Therefore, even though the extension from polynomials to formulas with any exponent is not enough to define analytically all the functions constructible by tractional motion, the construction of the exponential was important to focus on the role of the wheel for the expansion into the synthetic aspect.

D.A. functions are solutions of differential polynomials. Differential polynomials are polynomials in the variables and their derivatives, but these derivatives have to be of non-negative integer order. Negative integer-order derivatives can be considered integrals. However, what does it mean to consider derivatives of non-integer order? This question is older than three centuries and is at the core of \textit{fractional calculus}.

The idea of extending the meaning of $\frac{d^n y}{dx^n}$ to $n\notin \mathbb{N}$ appeared the first time in a letter of Leibniz to L'H\^{o}pital (September 30, 1695), and later got many mathematicians interested in it: 	Euler, 	Fourier, 	Abel, 	Liouville,	Riemann, 	Laurent,	Hadamard, 	Schwartz (for more precise historical references see 
\cite{Ros1975}, \cite{Ros1977}).
Nowadays, fractional calculus finds use in many fields of science and engineering, including fluid flow, rheology, diffusive transport akin to diffusion, electrical networks, electromagnetic theory, and probability.
There is not a unique definition of fractional integral, but the following (usually called \emph{Riemann-Liouville fractional integral}) is probably the most used version. We are proposing it just to give a first shallow idea, for clarifications and further reading, see \cite{MR1993}.

Starting from the Cauchy formula for repeated integration (it allows to compress $n$ antidifferentiations of a function into a single integral)
$$D^{-n}f(x) = \frac{1}{(n-1)!} \int_a^x\left(x-t\right)^{n-1} f(t) dt,$$
we can generalize $n$ to non-integer values and, since $n!=\Gamma(n+1)$, we get
$$D^{-v}f(x) = \frac{1}{\Gamma(v)} \int_a^x\left(x-t\right)^{v-1} f(t) dt.$$
This formula links $\Gamma$ function and fractional calculus.
The construction of $\Gamma$ with idealized machines could be particularly important because a widely accepted geometric interpretation of fractional calculus is still missing (for some attempts, see \cite{Add1997, Pod2002, TTR2013, Her2014}). Hence, from a historical/philosophical perspective, fractional calculus is now looking for a constructive-synthetic geometrical legitimation, as it happened in early modern period with transcendental curves. We hope that TMMs can constitute a solid step over which such extensions may come.

\subsection{Further perspectives}
Since Newton and Leibniz, the core concept of calculus is the constructive role of methods involving the infinity.
On the contrary, the proposed mechanical setting and the differential algebra counterpart suggest that it is possible to consider calculus (at least the part dealing with differential polynomials) without the need of infinity, 
but with the 
less abstract idea that ``the wheel direction is the tangent.'' 
A pedagogical peculiarity of such an approach to infinitesimal analysis with mathematical machines is that students can manipulate concepts usually considered too abstract. Some very preliminary attempts to introduce tractional machines in math education have been proposed in a workshop \cite{MDP2012} and in some papers \cite{Mil2012UCNC, DPM2012, SM2013}. A proposal for science museum activities can be also found in \cite{MD2012}.

Furthermore, the possibility of a restructuration of infinitesimal analysis in the light of TMMs and differential algebra should be interesting to be investigated from computational, instrumental, visual, algebraic, cognitive, and foundational viewpoints \cite{Mil2017}.
\\


Finally, we want to conclude with a remark on the mutual evolution of analog and digital/symbolic computation in mathematics.
An approach to break the Church-Turing thesis is to check whether some
results beyond Turing computational limits may be reached somehow (the \textit{hypercomputation} problem, e.g. \cite{Cop2002}). With regard to this question, it could be interesting to set the problem from a purely mathematical point of view.
Instead of considering the physical limits of analog computing, one could look for
an ``exact'' approach to analog computation through geometry because of its cognitive simplicity and richness. From this point
of view, considering diagrammatic constructions and symbolic manipulations
respectively as analog and digital computations, the evolution of mathematical
foundational paradigms from the geometric/arithmetic perspectives (with
their relative intercourses and extensions) can be considered an evolution of
computational limits.

Considering the computational power of mathematical approaches, Pythagorean
rational numbers (arithmetic perspective) were not sufficient to express some values generated by the arithmetic reinterpretation of ruler-and-compass geometric constructions (the length of the diagonal of a square with respect to the edge). On the contrary, later polynomial
algebra introduced values not geometrically constructible by ruler and compass
(the exactness problem in the early modern period). However, the unbalance
between the powers of the different paradigms is not a constant. Descartes balanced
their powers in analytical geometry, and this powerful paradigm became
the hard core over which calculus evolved, generating a rich symbolism inspired
by ideas derived from geometry and mechanics. Something new happened with
regard to calculus: if the geometrical paradigm had already been abandoned
in other periods, for the first time there was the acceptance of entities generated by infinite processes (note that infinite procedures were also adopted by Archimedes, but only as an investigative tool to be later interpreted from a synthetic perspective).
This acceptance of infinite processes made it difficult to re-interpret the
obtained entities in everyday (finite) experience.
Hence the claim of this paper: we reached part of infinitesimal calculus with suitable
geometrical constructions (synthesis) and symbolic tools provided by a finite
algebra (analysis).

Even if the balance introduced has no claim of constructing something beyond
Turing limits, it is another case in which analog and digital constructive
powers are balanced (as in Descartes' geometry). As differential calculus
evolved on Cartesian geometry, we hope that in future it could be interesting to
investigate whether the new balance proposed in this work could become a step
for new computational paradigms and mathematical constructive approaches
beyond the limits of today.

\bibliographystyle{plainnat} 
\bibliography{000-paper}

\end{document}